\documentclass{amsproc}
\usepackage{euscript, graphicx, epstopdf}
\usepackage{cases}
\usepackage{mathrsfs}
\usepackage{bbm}
\usepackage{amssymb}
\usepackage{txfonts}
\usepackage{amscd}
\usepackage{amsfonts,latexsym,amsmath,amsxtra,mathdots,amssymb,latexsym,mathabx,mathtools}
\usepackage[all,cmtip]{xy}
\usepackage{color}
\usepackage{colordvi}
\usepackage{multicol}
\usepackage{hyperref}
\usepackage{tikz}
\usepackage{float}
\usepackage{setspace, floatflt}

\usepackage{tensor}
\usepackage[normalem]{ulem}

\allowdisplaybreaks

\newcommand{\BC}{{\mathbb {C}}}

\newcommand{\GL}{{\mathrm {GL}}} 
\newcommand{\SL}{{\mathrm {SL}}} 
\newcommand{\SO}{{\mathrm{SO}}}
\newcommand{\SU}{{\mathrm{SU}}} \newcommand{\Tr}{{\mathrm{Tr}}}
\newcommand{\Sp}{{\mathrm{Sp}}}

\newcommand{\Gal}{\mathrm{Gal}} 
\newcommand{\res}{\mathrm{Res}}
\newcommand{\ind}{\mathrm{Ind}}

\newcommand{\ra}{\rightarrow}

\def\-{^{-1}}

\def\BCx{\BC^{\times}}

\def\shskip{\hskip 0.5 pt}

\newcommand{\mtrix}[4]{\left(\begin{matrix}	#1  & #2  \\  #3  &  #4   \end{matrix}\right)}

\makeatletter
\g@addto@macro\normalsize{\setlength\abovedisplayskip{3pt}}
\makeatother

\makeatletter
\g@addto@macro\normalsize{\setlength\belowdisplayskip{3pt}}
\makeatother

\newcommand{\delete}[1]{}

\theoremstyle{plain}

\newtheorem{thm}{Theorem}[section] \newtheorem{cor}[thm]{Corollary}
\newtheorem{lem}[thm]{Lemma}  \newtheorem{prop}[thm]{Proposition}

 \newtheorem{defn}[thm]{Definition}

\newtheorem {rem}[thm]{Remark} \newtheorem {example}[thm]{Example}

\numberwithin{equation}{section}

\begin{document}

	\title{Distinguished representations, Shintani base change and a finite field analogue of a conjecture of Prasad}

	\author{Chang Yang}
	\address{Key Laboratory of High Performance Computing and Stochastic Information Processing (HPCSIP)\\ College of Mathematics and Statistics \\ Hunan Normal University \\Changsha,  410081\\China}
	\email{cyang@hunnu.edu.cn}

%	\subjclass[2010]{22E50, 33C10}
	%\thanks{The first author is supported by the National Natural Science Foundation of China [Grant 11771131].}

	\begin{abstract}
	Let $E/F$ be a quadratic extension of fields, and $G$ a  connected quasi-split reductive group over $F$. Let $G^{op}$ be the opposition group obtained by twisting $G$ by the duality involution considered by Prasad. Assume that the field $F$ is finite. Let $\pi$ be an irreducible generic representation of $G(E)$. When $\pi$ is a Shintani base change lift of some representation of $G^{op}(F)$, we give an explicit nonzero $G(F)$-invariant vector in terms of the Whittaker vector of $\pi$. This shows particularly that $\pi$ is $G(F)$-distinguished.
		
		When the field $F$ is $p$-adic, the paper also proves that the duality involution takes an irreducible admissible generic representation of $G(F)$ to its contragredient. As a special case of this result, all generic representations of $G_2,\ F_4$ or $E_8$ are self-dual.

	%	For a connected reductive algebraic group $G$ over $F$, a finite field of $q$ elements, let $G^{op}$ be the %opposite group over $F$ that is constructed by Prasad in his conjectures describing distinguished %representations for Galois pairs over $p$-adic fields. We prove: Let $\pi$ be an irreducible representation of %$G(E)$, where $E$ denote a quadratic extension of $F$. Assume that $\pi$ is generic with respect %to a non-degenerate character of $U(E)$ that is trivial on $U(F)$, with $U$ a maximal unipotent subgroup of $G$. %If $\pi$ is a Shintani base change lift of some representation of $G^{op}(F)$, then $\pi$ is distinguished by %$G(F)$. 
	\end{abstract}
	
	\maketitle
	
\section{Introduction}

Let $G$ be a connected reductive algebraic group over a field $F$. Let $H$ be a subgroup of $G(F)$. A complex representation $(\pi, V)$ of $G(F)$ is said to be $H$-distinguished or distinguished by $H$ if there is a nonzero linear functional $l$ on $V$ such that $l( \pi (h) v)  =  v$ for all $h \in H$ and $v \in V$. 

One particular case is that $H$ is the fixed points of a Galois involution, that is, $(G,H) = (G(E),G(F))$, where $E/F$ is a separable quadratic extension of fields. When $F$ is a local field, D. Prasad \cite{Prasad-Relative+LLC} has proposed very general conjectures on the classification of irreducible admissible representations of $G(E)$ which is $G(F)$-distinguished, and the dimension of these $G(F)$-invariant linear forms, in terms of the Langlands parameter of the representation. See \cite{Beuzart-Plessis-Distinguish-Galois}, \cite{ZhangChong-Distinction-Depth0}, \cite{LuHengfeng-Prasad+Conj-U2} and \cite{LuHengfeng-Prasad+Conj-GSp4} for some recent progress toward these conjectures, and also works of Anandavardhanan and Prasad \cite{Anandavardhanan-Prasad-Distinguish-SL2,Anandavardhanan-Prasad-Peroid-SL2,Anandavardhanan-Prasad-Distinguish-SLn}
on distinction problems for $\SL_n$ that motivate the general conjecture.

In this paper we are concerned mainly with the finite field case. Let $E/F$ be a quadratic extension of finite fields. It is proved by Prasad \cite{Prasad-Quadratic-Compositio} that for any connected reductive group $G$ over $F$, an irreducible, uniform representation $\pi$ of $G(E)$ is distinguished by $G(F)$ if and only if $\pi^{\sigma} \cong \pi^{\vee}$, where $\sigma$ is the Frobenius map associated to the $F$-rational structure of $G$, $\pi^{\sigma}$ is the representation defined by $\pi^{\sigma}(g) = \pi (\sigma(g))$, and $\pi^{\vee}$ is the contragredient of $\pi$. Moreover, the space of $G(F)$-invariant vectors in $\pi$ is at most one dimensional. Recall that a representation is uniform if it is a virtual sum of Deligne-Lusztig representations. This generalizes earlier results obtained by Gow \cite{Gow-1984} for $G  =  \GL_n$ or $\mathrm{U}_n$.

A natural question is to describe explicitly the unique, up to scalars, $G(F)$-invariant vector in some models of a distinguished representation. This question was answered by Anandavardhanan and Matringe in \cite{Anan***-Matringe-BaseChange} for generic representations of $G(E)$ when $G = \GL_n$ or $\mathrm{U}_n$ (see Theorem 1.1, Corollary 1.2 in \emph{loc.cit.}). The purpose of this work is to generalize their results to general connected reductive groups.

In this paper we deal with distinguished representations which are also generic. Let $B$ be a Borel subgroup of $G$ over $F$ and $U$ be its unipotent radical. Let $\psi$ be a nondegenerate character of $U(F)$. A representation $\pi$ of $G(F)$ is said to be $\psi$-generic if it admits a nonzero $\psi$-Whittaker functional. Such functionals form a vector space of dimension at most one by a theorem of Steinberg \cite{Steinberg-Yale-1967}. The unique realization of $\pi$ into the space of functions on $G(F)$ that are left $(U(F),\psi)$-invariant is called the $\psi$-Whittaker model of $\pi$, denoted by $\EuScript W(\pi,\psi)$. Note that, in the Whittaker model of a $\psi$-generic representation $\pi$, there is a unique, up to scalars, function $B_{\pi,\psi}$ that is also right $(U(F),\psi)$-invariant, called the $\psi$-Bessel function of $\pi$. Note also that $\psi$-Bessel functions of $\pi$ are just $\psi$-Whittaker vectors in the Whittaker model.

To state our main result, let us introduce more notions. It is well known that the class of distinguished representations can usually be characterized as base change lifts from another group. By the theory of base change for $\GL_n(F)$ due to Shintani \cite{Shintani} and for $\mathrm{U}_n(F)$ due to Kawanaka \cite{{Kawanaka-Unitary+ShintaniLift}}, it follows from the work of Gow \cite{Gow-1984} that an irreducible representation of $\GL_n(E)$ is $\GL_n(F)$ resp. $\mathrm{U}_n(F)$-distinguished if and only if it is a base change lift of some representation of $\mathrm{U}_n(F)$ resp. $\GL_n(F)$. For a general connected reductive group $G$ over $F$, representations of $G(E)$ that are distinguished by $G(F)$ should arise as base change lifts from a group $G^{op}(F)$. The \emph{opposition group} $G^{op}$ is also a connected reductive group which is obtained by twisting $G$ by a duality involution $\iota_{G}$. We remark that the notions of duality involution and opposition group are valid for an arbitrary field, and also that these notions depend on an extra structure, an $F$-pinning, which will be made clear in the main body of the paper (see \cite{Prasad-Relative+LLC} for the role of $G^{op}$ in those conjectures of Prasad in his ``relative Langlands correspondence").

We next introduce the notion of \emph{Shintani base change} in the representation theory of finite reductive groups \cite{Kawanaka-Shintani-Arcata}. For a general connected reductive group $G$, the Shintani base change for $G(F)$ is a correspondence between irreducible representations of $G(F)$ and irreducible, $\sigma$-invariant representations of $G(E)$ which is given by character identities (see \cite{Kawanaka-Shintani-Arcata} or Section \ref{section::Shintani}). When $G = \GL_n$ or $\mathrm{U}_n$, every irreducible, $\sigma$-invariant representation of $G(E)$ is the base change lift of a unique representation of $G(F)$ by works of Shintani \cite{Shintani} and Kawanaka \cite{Kawanaka-Unitary+ShintaniLift}. There is also a fairly well understanding of base change lifts when the representations involved are uniform by the work of Digne-Michel \cite{Digne-Michel-Shintani+Lifting-Uniform}. We mention that the only property of Shintani base change that matters for us is the character identity in its definition.

Our main result is as follows.

\begin{thm}\label{theorem::Main}
    Let the groups $G$, $B$ and $U$ be as above. Let $\psi$ be a nondegenerate character on $U(E)$ that is trivial on $U(F)$, and let $G^{op}$ be the opposition group corresponding to $\psi$ \footnote{As we shall see later, every such nondegenerate character arises from an $F$-pinning of $G$.}. Let $\pi$ be an irreducible, $\psi$-generic representation of $G(E)$. Assume that $\pi$ is a base change lift of a representation $\rho$ of $G^{op}(F)$. Let $\lambda$ be the $G(F)$-invariant linear functional on the Whittaker model $\EuScript W(\pi,\psi)$ of $\pi$ defined by 
	\begin{align*}
	\lambda (W)  =  \sum_{h \in G(F)}  W (h),
	\end{align*}
	then 
	\begin{align}\label{formula::Main Theorem}
	\lambda (B_{\pi,\psi})  =\varepsilon(\pi) \frac{\dim \rho}{\dim \pi} \frac{|G(E)|}{|G^{op}(F)|},
    \end{align}
	where $B_{\pi,\psi}$ is the normalized Bessel function associated to $\pi$ and $\psi$, and $\varepsilon(\pi) = \pm 1$ is the quantity in the definition of Shintani base change (see Definition \ref{defn::Shantani bc} and Remark \ref{rem::Shintani bc}). In particular, $\pi$ is distinguished by $G(F)$.
\end{thm}

\begin{cor}
	Assume further that $\pi$ is uniform. Then the unique, up to scalars, $G(F)$-invariant vector in the Whittaker model $\EuScript W(\pi,\psi)$ of $\pi$ is given by
	\begin{align*}
	   \sum_{h \in G(F)} \pi (h)B_{\pi,\psi}.
	\end{align*}
\end{cor}

	By \eqref{formula::Main Theorem}, the linear form $\lambdaup$ is non-vanishing on the Bessel function $B_{\pi,\psi}$. Such a vector (in the Whittaker space) is called a test vector for the linear form. As remarked in \cite{Anan***-Matringe-BaseChange}, the formula \eqref{formula::Main Theorem} shows that the relationship between base change and distinciton lies deeper and it reflects at the level of a test vector. We also remark that the conclusion of the theorem that $\pi$ is $G(F)$-distinguished can be covered by the result in \cite{Prasad-Quadratic-Compositio}, at least when $\pi$ is uniform; we will discuss this in Section \ref{section::remark}.

We will now briefly describe the method of the proof. As in \cite{Anan***-Matringe-BaseChange}, the formula \eqref{formula::Main Theorem} is approached by comparing the linear form $\lambdaup$ that gives distinction by $G(F)$ with a Whittaker linear form defined in \eqref{formula::Whittaker lin-form}. The key issue for this comparison is to understand which double cosets $U(F) \backslash G(F)  / U(F)$ support the Bessel function, which is the main bulk of the paper.

Here we state another result of this paper. Let $G$ be a quasi-split connected reductive group over a field $F$, which is either a finite field or a $p$-adic field. Let $\psi$ be a nondegenerate character of $U(F)$ with $U$ the unipotent radical of a Borel subgroup of $G$ over $F$. Let $\iota_{G}$ be the duality involution corresponding to $\psi$. The following result confirms a conjecture of Prasad \cite{Prasad-Involution} on the contragredient and the duality involution.

%%---------------
%%A critical role in the construction of $G^{op}$ is an involution $\iota_{G}$, called the duality involution. It is constructed %%by Prasad in \cite{Prasad-Involution} and generalizes the $\mathrm{MVW}$-involution for classical group \cite{M-V-W}. We remark %%that the duality involution had already appeared in the literature in a sketchy or less transparent way, such as %%\cite[Appendix]{Shalika-MO} or \cite{Steinberg-Yale-1967}.
%%----------------------

\begin{thm}
	Let notations be as above. Let $\pi$ be an irreducible admissible representation of $G$. Assume that $\pi$ is $\psi$-generic. Then
	\begin{align}
	\pi^{\iota_G} \cong \pi^{\vee}.
	\end{align}
\end{thm}

For exceptional groups $G_2$, $F_4$ and $E_8$, each duality involution is an inner conjugation, hence we have the corollary.

\begin{cor}
	Every generic representation of $G_2$, $F_4$ or $E_8$ is self-dual.
\end{cor}

This paper is organized as follows. In Section \ref{section::Preli}, we introduce some related notations and preliminaries most in the general theory of reductive groups, particularly the notion of $F$-pinning, on which the duality involution and the opposition group will depend. In Section \ref{section::Involution}, we review the construction of the duality involution over a general field that is given by Prasad in \cite{Prasad-Involution} and describe explicitly its action on representatives of Weyl elements. In Section \ref{section::Bessel}, we study Bessel functions associated to generic representations of a finite reductive group. As a byproduct, we verify that the composition of the duality involution with the inversion is the anti-automorphism used by Shalika in his famous multiplicity one paper \cite{Shalika-MO}. Section \ref{section::opposition Shintani} is devoted to the notion of opposition group and Shintani base change. In Section \ref{section::Main}, we prove Theorem \ref{theorem::Main}. In Section \ref{section::remark}, the relation between duality involution and contragredient is discussed.

%There are already some works on the distinction problem in the case of finite fields. For connected reductive groups $G$ %with connected center, Lusztig in \cite{Lusztig-Distinguish} gave an explicit calculaiton for the muliplicity $m_2 (\pi) = %\dim_{\BC} \Hom_{G(F)}(\pi,\BC) $ for all irreducible representations $\pi$ of $G(E)$. Before Lusztig's work, %Kawanaka in \cite{Kawanaka-Galois+Symmetric} gave a formula of $m_2$ for all irreducible representations of classical %groups. For groups with non-connected center, Shoji and Sorlin in \cite{Shoji-SLn+Distinguish} gave an explicit computation %for finite special linear group. All the works are based on a general formula for the multiplicty of a Deligne-Lusztig %character in \cite{Lusztig-Symmetric+FiniteFields} and a relation between the multiplicity formula and twisted %Frobenius-Schur indicators. However, it is not clear to see the relationship between distinction property and functorial %lifts from their results.

\section{General notations and preliminaries}\label{section::Preli}

For a set $X$, let $\lvert X \rvert$ denote its cardinality. If $\sigma$ is a transformation of $X$, let $X^{\sigma}$ denote the set of elements of $X$ fixed by $\sigma$. For a group $G$ and $g \in G$, let $C_G(g)$ denote the conjugacy class of $g$ in $G$.

Assume, unless otherwise stated, that the field $F$ is arbitrary. Fix a separable closure $F_{s}$ of $F$ and let $\Gamma_F = \Gal(F_s / F)$ be the Galois group of $F$. 

Let $G$ be a quasi-split connected reductive algebraic group over $F$. We often identify $G$ with its $F_s$-points. Let $B = TU$ be a Borel subgroup defined over $F$ with unipotent radical $U$ and maximal torus $T$. Let $S$ be the maximal $F$-split subtorus of $T$. One has $T = Z_G (S)$.

    Let $\widetilde{\Phi}$ be the set of absolute roots of $G$ with respect to $T$. Let $\widetilde{\Phi}^+$ and $\widetilde{\Delta}$ be the subsets of positive roots and simple roots associated to $B$. Let $\Phi$ be the set of relative roots of $G$ with respect to $S$. Let $\Phi^+$ and $\Delta$ be the subsets of positive roots and simple roots with respect to the order that is compatible with the order in $\widetilde{\Phi}$. The root system $\Phi$ can be non-reduced. For $ a \in \Phi$, let $(a)$ denote all possible multiples of $a$ in $\Phi$. One has $(a) = \{a\}$ or $\{a, 2a\}$.  
    
% Let $\tilde{k} / k $ be the Galois splitting field of $G$ and let $\Gamma_{\tilde{k}} = \Gal(\tilde{k} / k)$. Since %$G$ is defined over $F$ and splits over $\tilde{k}$, the group $\Gal(F_s /  \tilde{k})$ acts trivially on %$\widetilde{\Phi}^+$ and $\widetilde{\Delta}$, hence both $\widetilde{\Phi}^+$ and $\widetilde{\Delta}$ decompose %into a finite number of $\Gamma_{\tilde{k}}$ orbits.

%      For $\alpha \in \Phi$, denote by $U_{\alpha}$ the root subgoup of $G$ corresponding to $\alpha$ over $F$.

Let $\widetilde{W} = N_G(T) / T$ be the absolute Weyl group, and $W = N_G(S) / T$ be the relative Weyl group. Since $Z_G(S)  =  T$, one can view $W$ as a subset of $\widetilde{W}$. Denote by $w_0$ the longest Weyl element in $\widetilde{W}$. Then $w_0$ is also the longest Weyl element in $W$. 

%For each $\alpha \in \Phi$, let $w_{\alpha}$ denote the corresponding reflection on $X(T)$ which we extend to %$X(T)\otimes_{\BZ}\BR$. Then $W$ is isomorphic to the finite reflection group generated by all $w_{\alpha}$, $\alpha 
%\in \Phi$. We will view $w_{\alpha}$ as elements in $W$. As an abstract group, $W$ is a Coxeter group with %generators $w_{\alpha}$, $\alpha \in \Delta$. Denote by $l$ the length function on $W$.

%  Denote by $w_0$ the unique longest Weyl element in $W$. We have $w_0^2 = 1$ and $w_0 (\Phi^+)  =  - \Phi^+$. For %$I \subset \Delta$, let $\Phi (I) \subset \Phi$ be the roots generated by $I$. The group $L_I$ generated by $T$ and %$U_{\alpha}$, $\alpha \in \Phi(I)$ is split and reductive. Let $W_I = \langle w_{\alpha} \ |\ \alpha \in I \rangle$. %Then $W_I$ is the Weyl group of $L_I$. Denote by $w_I$ the longest Weyl element in $W_I$.

\subsection{$F$-Pinnings}\label{subsection::pinning}

We recall some facts about the structure of root subgroups of a quasi-split reductive group. Our references are \cite[\S 4.1]{Bruhat-Tits-RG+LocalField-II} and \cite[\S 1.1]{Cogdell-PS-Shahidi-PartialBessel}. For $\alpha \in \widetilde{\Phi}^+$, let $\tilde{U}_{\alpha}$ be the corresponding root subgroup of $G$. Let $L_{\alpha}$ be the \emph{splitting field} of $\alpha$, that is, the stabilizer of $\alpha$ in $\Gamma_F$ is given by the subgroup $\Gal(F_s / L_{\alpha})$. The action of $\Gal(F_s / L_{\alpha})$ on $\tilde{U}_{\alpha}$ provides by Galois descent a $L_{\alpha}$-rational structure on $\tilde{U}_{\alpha}$.

When $B$ and $T$ are fixed, by an $F$-\emph{pinning} of $G$ we mean a set $ \{\tilde{x}_{\alpha}\ |\  \alpha \in \widetilde{\Delta} \}$ such that
  \begin{itemize}
  \item[(i)]$\ $ $\tilde{x}_{\alpha} \colon \mathbb{G}_a \ra \tilde{U}_{\alpha} \text{ is an isomphism over }L_{\alpha}$, for all $\alpha \in \widetilde{\Delta}$,
  \item[(ii)]$\ $ $\tilde{x}_{\gamma(\alpha)}  = \gamma \circ \tilde{x}_{\alpha} \circ \gamma^{-1}$, for all $\alpha \in \widetilde{\Delta} \text{ and }\gamma \in \Gamma_{F}$. 
  \end{itemize}
  If $B$ and $T$ are taken into consideration, by an $F$-pinning of $G$ we mean a set $\mathcal{P} = \{G,B,T,\{ \tilde{x}_{\alpha}\}   \}$ with $\{\tilde{x}_{\alpha}\}$ defined as above. One can extend this $F$-pinning to a \emph{Chevalley-Steinberg system} $\{\tilde{x}_{\alpha} \ |\ \alpha \in \widetilde{\Phi} \}$ of $G$ (see \cite[\S 4.1.3]{Bruhat-Tits-RG+LocalField-II}).
  
  For $a \in \Phi$, let $U_a$ be the corresponding root subgroup of $G$ defined over $F$. For $a \in \Phi^+$, let $G^a$ be the simply connected covering of the derived group of the rank-one subgroup of $G$ associated to $a$. If $(a) = \{a \}$, take $\alpha \in \widetilde{\Phi}^+$ be a root of $T$ restricting to $a$. Then $G^a \cong \res_{L_{\alpha} / F} \SL_2$. In fact, there exists a unique $L_{\alpha}$-homomorphism $\zeta_{\alpha}$ of $\SL_2$ into $G$ such that 
  \begin{align}
  \zeta_{\alpha} \mtrix{1}{u}{0}{1} = \tilde{x}_{\alpha} (u)\quad\text{and}\quad\zeta_{\alpha} \mtrix{1}{0}{-u}{1} = \tilde{x}_{-\alpha} (u).
  \end{align}
  Then the $F$-homomorphism $\zeta_a = \res_{L_{\alpha} /F} \zeta_{\alpha}$ of $\res_{L_{\alpha} / F}\SL_2$ into $G$ is a universal covering of the derived group of the rank-one subgroup of $G$ associated to $a$. The map $x_a = \res_{L_{\alpha} / F} \tilde{x}_{\alpha}$ is an $F$-isomorphism from $\res_{L_{\alpha} / F}\mathbb{G}_a$ onto $U_a$. If $u \in L_{\alpha}$, then
  \begin{align}\label{formula::x_a Case I}
  x_a (u) = \prod_{ \beta \in \widetilde{\Delta}_a} \tilde{x}_{\beta}(u_{\beta}) \quad\text{with }u_{\gamma (\alpha)} = \gamma (u) \quad \text{for }\gamma \in \Gamma_{F}.
  \end{align}
  If $(a) = \{a,2a \}$, take $\alpha \in \widetilde{\Phi}^+$ be a root of $T$ restricting to $a$. Let $\alpha^{\prime}$ be the unique simple root in $\widetilde{\Phi}^+$ restricting to $a$ such that $\alpha + \alpha^{\prime}$ is still a root. We then have that $L_{\alpha}  =  L_{\alpha^{\prime}}$ and $L_{\alpha}$ is a quadratic separable extension of $L_{\alpha + \alpha^{\prime}}$. Let $L = L_{\alpha}$, $L_2 = L_{\alpha +\alpha^{\prime}}$ and $\sigma$ be the non-trivial element of $\Gal(L/L_2)$. Then $G^a \cong \res_{L_2 /F} \SU_3$. In fact, there exists a unique $L$-homomorphism $\zeta_{\{\alpha,\alpha^{\prime}\}}$ of $\SL_3$ into $G$ such that
  \begin{align*}
  \zeta_{\{\alpha,\alpha^{\prime}\}} \left(
        \begin{matrix}
           1 & w & -v \\ 0 & 1 & u \\ 0 & 0 & 1
        \end{matrix}  
        \right)
  &=  \tilde{x}_{\alpha}(u) \tilde{x}_{\alpha + \alpha^{\prime}} (-v) \tilde{x}_{\alpha^{\prime}} (-w),  \\
  \zeta_{\{\alpha,\alpha^{\prime}\}} \left(
  \begin{matrix}
  1 & 0 & 0 \\ w & 1 & 0 \\ -v & u & 1
  \end{matrix}  
  \right)
  &=  \tilde{x}_{-\alpha^{\prime}}(w) \tilde{x}_{-(\alpha + \alpha^{\prime})} (v) \tilde{x}_{-\alpha} (-u).
  \end{align*}
  We then transport the action of $\sigma$ to $\SL_3$ via $\zeta_{\{\alpha,\alpha^{\prime}\}}$ and get the $\SU_3$ defined over $L_2$ by twisting $\SL_3$ by $\sigma$ (see \cite[\S 4.1.9]{Bruhat-Tits-RG+LocalField-II}). Then the $F$-homomorphism $\zeta_a = \res_{L_2/ F} \zeta_{\{\alpha,\alpha^{\prime}\}}$ of $\res_{L_2 / F}\SU_3$ into $G$ is a universal covering of the derived group of the rank-one subgroup of $G$ associated to $a$. The upper triangular unipotent matrices in $\SU_3$ are of the form
  \begin{align*}
  \mu (u,v) = \left(\begin{matrix}
                1 & -u^{\sigma} & -v \\ 0 & 1 & u \\ 0 & 0 & 1  
              \end{matrix} \right)
  \end{align*}
  for $u$, $v \in L$ such that $v+v^{\sigma} = u^{\sigma}u$. Let $H_0(L,L_2)$ denote the subvariety of $L \times L$, considered as a vector space of dimension $4$ over $L_2$, defined by the equation  $v + v^{\sigma}  =  u^{\sigma} u$. Then $H_0(L,L_2)$ is a group over $L_2$ with the multiplication given by
%The group law on $H_0(L,L_2)$ which is given by
  \begin{align*}
  (u,v) (u^{\prime},v^{\prime})  =  (u+u^{\prime}, v+v^{\prime} + u^{\sigma} u^{\prime}).
  \end{align*}
  The map $x_a = \res_{L_2 /F} (\zeta_{\{\alpha,\alpha^{\prime}\}} \circ \mu)$ is an $F$-isomorphism of $\res_{L_2 /F} H_0(L,L_2)$ onto $U_a$. If $(u,v) \in L \times L$ satisfying $v + v^{\sigma}  =  u^{\sigma} u$, then
  \begin{align}\label{formula::x_a Case II}
  x_a (u,v) = \prod \tilde{x}_{\beta} (u_{\beta}) \tilde{x}_{\beta + \beta^{\prime}} (-v_{\beta}) \tilde{x}_{\beta^{\prime}} ( u^{\sigma}_{\beta}),
  \end{align}
  where the product is over distinct pairs $\{ \shskip \beta, \beta^{\prime}\}$ such that $\beta$, $\beta^{\prime}$ restrict to $a$ and $\beta + \beta^{\prime}$ is a root, and for each $\beta$, choose $\gamma \in \Gamma_{F}$ such that $ \beta = \gamma (\alpha)$, we have $\beta^{\prime} = \gamma (\alpha^{\prime})$, $u_{\beta} = \gamma (u)$ and $v_{\beta} = \gamma (v)$. Note that the image of $x_a (u,v)$ in $U_a / U_{2a}$ depends only on $u$ and will be denoted by $\bar{x}_a(u)$. The map $u \mapsto \bar{x}_a(u)$ gives an $F$-isomorphism of $\res_{L/F}\mathbb{G}_a$ onto $U_a / U_{2a}$. Set 
  \begin{align}\label{formula::defn-Va}
  V_a  =  U_a / U_{2a},
  \end{align}
  where we take $U_{2a}$ to be trivial if $2a$ is not a root. Then in all the two cases above, we have an isomorphism
  \begin{align}\label{formula::Va}
  \res_{L_{\alpha}/F} \mathbb{G}_a  \cong V_a
  \end{align}
  defined over $F$ with $\alpha$ a root of $T$ restricting to $a$.

\subsection{Canonical representatives of Weyl elements} \label{section::Weyl lifting}

Given an $F$-pinning, we have a canonical choice of representatives of Weyl group elements (see \cite{Bruhat-Tits-RG+LocalField-II}, \cite{Shahidi-LocalCoefficients}). 

For $\alpha \in \widetilde{\Delta}$, let $w_{\alpha} \in \widetilde{W}$ be the Weyl element associated to the simple reflection for $\alpha$. Set
  \begin{align}\label{formula::Weyl lifting-absolute}
  n_{w_{\alpha}}  =  \tilde{x}_{\alpha}(1) \tilde{x}_{-\alpha}(1) \tilde{x}_{\alpha}(1).
  \end{align}
  Then $n_{w_{\alpha}}$ lies in $N_G(T)$ and has image $w_{\alpha}$ in $\widetilde{W}$. We can choose representatives for each $w \in \widetilde{W}$ by means of a reduced decomposition and this choice of $n_{w_{\alpha}}$. This is independent of the choice of the reduced decomposition (see \cite[\S 9.3.3]{Springer-LAG}).
  
%If $w = w_{\alpha_1}\cdots w_{\alpha_s}$ is a reduced decomposition of $w \in W$, the element $\dot{w}_{\alpha_1} %\cdots \dot{w}_{\alpha_s}$ is independent of the choice of the reduced decomposition of $w$. We denote this element %by $\dot{w}$.
  
For $a \in \Delta$, let $w_a \in W$ be the element associated to the simple reflection for $a$. \\Case I. If $(a) = \{a\}$, set
\begin{align}
n_{w_a} = \zeta_a \left(\left(   \begin{matrix}
                         0 & 1  \\ -1 & 0
                       \end{matrix}   \right)\right).
\end{align}
Case II. If $(a) = \{ a,2a \}$, set
\begin{align}
 n_{w_a}   =   \zeta_a \left( \left( \begin{matrix}
                                    0 & 0 & -1 \\ 0 & -1 & 0 \\ -1 & 0 & 0
                                 \end{matrix}  \right)  \right).
\end{align}
Then $n_{w_a}$ lies in $N_G(S)(F)$ and has image $w_a$ in $W$ (see \cite[\S 4.1.5, 4.1.9]{Bruhat-Tits-RG+LocalField-II}). In view of \eqref{formula::Weyl lifting-absolute} and the definition of $\zeta_a$, we have 
 \begin{align}
 n_{w_a}  =  \prod_{\alpha \in \widetilde{\Delta}_a} n_{w_{\alpha}} \quad \text{or} \quad n_{w_a}  =  \prod_{\{\alpha,\alpha^{\prime}\}} n_{w_{\alpha}} n_{w_{\alpha^{\prime}}} n_{w_{\alpha}},
 \end{align}
where, in the second equality, the product is over distinct pairs $\{\alpha , \alpha^{\prime} \} \subset \widetilde{\Delta}_a$ with $\alpha + \alpha^{\prime}$ a root. Note that $W$ can be viewed naturally as a subset of $\widetilde{W}$, and that 
\begin{align}\label{formula::Weyl group element-decomposition}
w_a = \prod w_{\alpha} \quad \text{or} \quad w_a = \prod w_{\alpha} w_{\alpha^{\prime}} w_{\alpha}
\end{align}
is a reduced decomposition of $w_a$ in $\widetilde{W}$. For a general $w \in W$, let $w = w_{a_1} \cdots w_{a_n}$ be a reduced decomposition of $w$ in $W$. One can check that, after decomposing each $w_{a_i}$ as in \eqref{formula::Weyl group element-decomposition}, we get a reduced decomposition of $w$ in $\widetilde{W}$. This implies that we can choose a representative $n_w \in N_G(S)(F)$ for each $w \in W$ by means of a reduced decomposition and those $n_{w_a}$', which is independent of the choice of the reduced decomposition, and that the two choices of representatives of $w$, viewed as an element in $W$ or in $\widetilde{W}$, coincide.

For a Weyl element $w$ in $W$ or $\widetilde{W}$, we denote by $ n_w $ the representative chosen above. 
\begin{lem}\label{lemma::representative of Weyl}
	If $a$, $b \in \Delta$ and $w\in W$ such that $w a = b$. Take $\alpha \in \widetilde{\Delta}$ that restricts to $a$ and let $\beta = w \alpha$, where we view $w$ as an element in $\widetilde{W}$. Then $L_{\alpha}  =  L_{\beta}$, and
		\begin{align*}
		n_w x_{a}(u)  n_w^{-1} = x_{b} (u), \quad u \in L_{\alpha},
		\end{align*}
		if $a = \{ a\}$ and $(b) = \{ b\}$, and
		\begin{align*}
		n_w \bar{x}_{a}(u) n_w^{-1} = \bar{x}_{b} (u), \quad u \in L_{\alpha},
		\end{align*}
		if $(a) = \{ a, 2a\}$ and $(b) = \{b,2b \}$.
\end{lem}
\begin{proof}
	This follows from \cite[Proposition, 9.3.5]{Springer-LAG} and the descriptions of $x_a$ and $\bar{x}_a$ in \eqref{formula::x_a Case I} and \eqref{formula::x_a Case II}.
	\end{proof}

\subsection{$F$-rational Bruhat decomposition}

For $w \in W$, let
  \begin{align*}
  \Phi^+_{w} = \{ a \in \Phi^+ \ |\ w(a) > 0 \} \quad\text{and} \quad \Phi^-_{w} = \{ a \in \Phi^+ \ |\ w(a) <0 \}.
  \end{align*}
  These are closed subsets of $\Phi$ (see \cite[\S 21.7]{Borel-LAG}). Let
  \begin{align}\label{formula::Bruhat decompostion }
  U^+_w = U_{\Phi^+_w} = U \cap w^{-1}Uw  \quad \text{and} \quad U^-_w = U_{\Phi^-_w} = U \cap w^{-1}U^-w
  \end{align}
  be the corresponding $F$-unipotent subgroups (see [\emph{loc.cit.}, \S 21.9]). Then
  \begin{align*}
  U = U^+_w U^-_w.
  \end{align*}
  
For $w \in W$ let $n_w \in N_G(S)(F)$ be the chosen representative above. Let $C(w) = B(F) n_w B(F) = B(F) n_w U^-_w(F)$ be the corresponding Bruhat cell. The $F$-rational Bruhat decomposition for $G$ is the decomposition of $G(F)$ as the disjoint union:
  \begin{align*}
    G (F)= \bigcup_{w \in W} C(w) = \bigcup_{w \in W} B(F)  n_w U^-_w(F).
  \end{align*}  
  As a simple consequence, we have $N_G(T)(F)  =  N_G(S)(F)$.

\subsection{Whittaker datum}\label{subsection::Whittaker}

  Set $U_1 = \prod_{ \alpha \in \widetilde{\Phi}^+, \alpha \notin \widetilde{\Delta}} U_{\alpha}$. Then $U_1$ is a normal $F$-subgroup of $U$. By \cite[\S 1.3 Proposition]{Bushnell-Henniart-DerivedSubgroup}, the quotient group $U/U_1$ is $F$-isomorphic to the direct product $\prod_{a \in \Delta} V_a$ (recall that the definition of $V_a$ is given in \eqref{formula::defn-Va}). Moreover, we have
  \begin{align}\label{formula::U/U1}
  U(F) / U_1(F)  \cong \prod_{a \in \Delta} V_a (F).
  \end{align}
  
  Following Carter \cite[\S 8.1]{Carter-1985}, we say that a character $\psi$ of $U(F)$ is nondegenerate if $\psi$ contains $U_1(F)$ in its kernel and its restriction to each $V_a(F)$ via \eqref{formula::U/U1} is nontrivial. As is shown by \cite[\S 4.1 Theorem]{Bushnell-Henniart-DerivedSubgroup}, when the field $F$ is a local field of characteristic not equal to $2$, the derived group of $U(F)$ is precisely $U_1(F)$, so this definition of nondegeneracy coincides with the usual one in the representation theory of $p$-adic groups.
%  The map $(u_1,u_2,\cdots ,u_s) \ra \prod_{1}^s u_i$ is an isomorphism from the direct product $U_{\alpha_1} \times \cdots %\times V_{\alpha_s}$ to the quotient group $U/U_1$.

By a \emph{Whittaker datum}, we mean a $G(F)$-conjugacy class of pairs $(B,\psi)$, where $B$ is a Borel subgroup of $G$ defined over $F$ with unipotent radical $U$, and $\psi$ is a nondegenerate character $U(F) \ra \BCx$. A representation is called generic, if it is generic with respect to some Whittaker datum.
  
  Let $\mathcal{P}$ be an $F$-pinning of $G$. We can associate naturally to $\mathcal{P}$ an $F$-homomorphism of algebraic groups $f \colon U \ra \mathbb{G}_a$ defined by
  \begin{align}\label{formula::algebraic Whittaker}
  f (\prod_{ \alpha \in \widetilde{\Phi}^+} x_{\alpha}  (u_{\alpha}))  =  \sum_{\alpha \in \widetilde{\Delta}} u_{\alpha}.
  \end{align}
  Fix a non-trivial additive character $\psi_0$ of $F$. We have a nondegenerate character $\psi_{\mathcal{P}}$ of $U(F)$ associated to the pinning $\mathcal{P}$ by setting $\psi_{\mathcal{P}} (u)  =  \psi_0 (f(u))$. If we identify $V_a(F)$ with $L_{\alpha}$ through \eqref{formula::Va}, it follows from \eqref{formula::x_a Case I} and \eqref{formula::x_a Case II} that the restriction of $\psi_{\mathcal{P}}$ to $V_a(F) \cong L_{\alpha}$ takes the form
  \begin{align}\label{formula::psi-P}
   \psi_{\mathcal{P}}  ( u )  =  \psi_0 (\Tr_{L_{\alpha} /F}\shskip u), \quad \text{for all } u \in L_{\alpha}.
  \end{align}
  When $F$ is a finite field or a $p$-adic field, every nondegenerate character of $U(F)$ arises in this way.
  
  The representatives $n_w$ of $w \in W$ chosen  with respect to an $F$-pinning $\mathcal{P}$ is compatible with the nondegenerate character $\psi = \psi_{\mathcal{P}}$ in the following sense. For a subset $I \subset \Delta$, let $[I]^+$ denote the set of $F$-roots which are both linear combinations of elements in $I$ and positive, and $U_I$ be the unipotent group corresponding to $[I]^+$.  For $w \in W$ such that $w I \subset \Delta$, we have that
  \begin{align}\label{formula:: compatible representative character}
  \psi (u)  =  \psi (n_w u n_w^{-1}),
  \end{align}
  for all $u \in U_I(F)$. This follows easily from Lemma \ref{lemma::representative of Weyl}.

\section{The duality involution $\iota_{G,\mathcal{P}}$}\label{section::Involution}

%Dipendra Prasad in \cite{Prasad-Involution} constructs an involution for any quasi-split reductive group over an %arbitrary field that generalizes the Moeglin-Vigneras-Waldspurger involution in \cite{M-V-W} for classical groups. %In the local field case, the involution is conjectured to take an irreducible admissible representation to its %contragredient (\cite[Conjecture 1]{Prasad-Involution}). The same involutions has already been considered by %Steinberg in \cite{Steinberg-Yale-1967} for quasi-split semisimple groups and by Shalika in \cite{Shalika-MO} for %general quasi-split groups when studying multiplicity one properties of Whittaker models.

In this section we review the duality involution constructed by Prasad in \cite{Prasad-Involution}. 

      Let $G$ be a quasi-split connected reductive group over an arbitrary field $F$ and $\mathcal{P}=(G,B,T,\{x_{\alpha} \})$ be a pinning of $G$ (not necessarily an $F$-pinning). Let $\Psi = (X,\Phi,X^{\vee},\Phi^{\vee};\tilde{\Delta})$ be the based root datum associated to $(G,B,T)$. Then there exists a triple $(G_0,B_0,T_0)$, with $G_0$ a split reductive group over $F$, $T_0$ an $F$-split maximal torus of $G_0$, and $B_0$ a Borel subgroup containing $T_0$, such that the groups $G_0(F_s)$, $B_0(F_s)$, $T_0(F_s)$ of $F_s$-rational points are the same as those of $G$, $B$, $T$, and that the based root datum associated to it is also $\Psi$. Assume that $G$ is twisted from $G_0$ by the cocycle
      \begin{align*}
      \lambdaup \colon \Gal(F_s / F)   \ra \mathrm{Aut}_{F_s}(G_0,B_0,T_0,\{x_{\alpha}\}),
      \end{align*}
      where $\mathrm{Aut}_{F_s}(G_0,B_0,T_0,\{x_{\alpha}\})$ is the group of $F_s$-automorphisms of $G_0$ that stabilize $B_0$, $T_0$ and the set of $x_{\alpha}$. Note that $-w_0$ gives an isomorphism between the based root datum $\Psi$ and itself. By \cite[Theorem 23.40]{Milne-AlgGroup}, there is a unique isomorphism $c_{G_0,\mathcal{P}_0}:G_0 \ra G_0$ defined over $F$ which fixes the pinning $\mathcal{P}_0 = (G_0,B_0,T_0,\{x_{\alpha} \})$ and induces $-w_0$ on the based root datum. The involution $c_{G_0,\mathcal{P}_0}$ lies in the center of $\mathrm{Aut}_{F_s} (G_0,B_0,T_0,\{x_{\alpha} \} )$. In particular, $c_{G_0,\mathcal{P}_0}$ commutes with the image of $\lambdaup$, hence defines an involution, to be called the \emph{Chevalley involution} associated to $\mathcal{P}$, $c_{G,\mathcal{P}} \colon G \ra G$ over $F$. It is clear that $c_{G,\mathcal{P}}  =  1$ if and only if $w_0  =  -1$, that is, $ -1 \in \widetilde{W}$.
      
      Since $\tilde{\Delta}$ is a basis of the character group of $T_0/Z_0$, with $Z_0$ the center of $G_0$, there exists a unique $t \in T_0/Z_0(F) $ such that $\alpha (t) = -1$ for all $\alpha \in \tilde{\Delta}$. Let $\iota_-  =  \mathrm{Inn}(t)$. One sees that $\iota_-$ commutes with every element in $\mathrm{Aut}_{F_s}(G_0,B_0,T_0,\{x_{\alpha}\})$, and hence defines an inner automorphism, still denoted by $\iota_-$, $\iota_- \colon G \ra G $ over $F$. By definition, $\iota_-$ acts by $-1$ on all simple spaces of $T$ in $B$.

     % As any automorphism in $\mathrm{Aut}_{F_s}(G_0,B_0,T_0,\{x_{\alpha}\})$ preserves the set of simple roots, one %sees that $t \in T/Z(F)$ when viewed as an element in $T_0/Z_0(F_s)= T/Z(F_s)$. Let $\iota_- = \mathrm{Inn}(t) %\in \mathrm{Aut}_k(G)$ be the conjugation by $t$. Then $\iota_-$ acts by $-1$ on all simple spaces of $T$ in %$B$ and commutes with $c_{G,\mathcal{P}}$. 

\begin{defn}(Duality involution)
	Let $G$ be a quasi-split reductive group over a field $F$ and $\mathcal{P}=(G,B,T,\{x_{\alpha} \})$ be a pinning of $G$. Let $c_{G,\mathcal{P}}$ and $\iota_-$ be the involutions constructed above. Define the duality involution $\iota_{G,\mathcal{P}} \in \mathrm{Aut}_F(G)$ as the product of two commuting involutions $\iota_{G,\mathcal{P}} = \iota_- \circ c_{G,\mathcal{P}}$.
\end{defn} 

\begin{rem}
	It is clear from the construction that the duality involution is the same for all quasi-split $F$-forms of a pinned split reductive group. In \cite{Prasad-Involution}, Prasad discusses further the dependence of $\iota_{G,\mathcal{P}}$ on the pinning $\mathcal{P}$. In particular, he shows that, under some mild conditions, the automorphism $\iota_{G,\mathcal{P}}$ is independent of the choice of the pinning (as an element of $\mathrm{Aut}_F(G)  /   G(F)$, here $G(F)$ is understood as the inner-conjugation by elements of $G(F)$), see \emph{loc.cit.} for a precise statement. However, this fact will not be used in our work, that is to say, we shall always work with a pinned group.
\end{rem}

   \begin{example}
       	For classical groups $G = \Sp_{2n},\  \SO_{2n}, \ \mathrm{U}_n$, the involution $\iota_{G,\mathcal{P}}$ coincides with the $\mathrm{MVW}$ involution \cite{M-V-W}. For exceptional groups $G_2,\ F_4$ and $E_8$, all these groups are simply-connected and adjoint, $c_{G,\mathcal{P}}  =  1$ and $t \in T(F)$.
   \end{example}      

      Note that $c_{G,\mathcal{P}} (t) = w_0 t^{-1} w_0$ for all $t \in T$, hence $c_{G,\mathcal{P}}(U_{\alpha})  =  U_{-w_0 \alpha}$. Since $c_{G,\mathcal{P}}$ stabilizes the set of $x_{\alpha}$, we have $c_{G,\mathcal{P}} (x_{\alpha} ( u ))  =  x_{-w_0\alpha} ( u )$, $u \in F_s$. After composing with $\iota_-$, we get that
\begin{align}\label{formula::Involution on Weyl  Basic}
\iota_{G,\mathcal{P}} (x_{\alpha} (u))  =  x_{-w_0\alpha} ( -u ),
\end{align}
for all $u \in F_s$ and $\alpha \in \widetilde{\Delta}$.

Assume that $\mathcal{P}$ is an $F$-pinning. Recall that in Section \ref{section::Weyl lifting}, we associated to $\mathcal{P}$ a canonical choice of representatives of elements of $\widetilde{W}$ in $N_G(T)$ and representatives of elements of $W$ in $N_G(T)(F)$. The following description of $\iota_{G,\mathcal{P}}$ will be central to our work.

\begin{thm}\label{theorem::Involution}
	For $w \in \widetilde{W}$, set $w^{\star}  =  w_0 w^{-1} w_0$. Let $n_w$ be the representative of $w$ with respect to $\mathcal{P}$ as in Section \ref{section::Weyl lifting}. Then
	\begin{align}
	\iota_{G,\mathcal{P}} ( n_w )  =  ( n_{w^{\star}} )^{-1}
	\end{align}
\end{thm}

\begin{proof}
	For $\alpha \in \tilde{\Delta}$, we have $w_{\alpha}^{\star}  =  w_0w_{\alpha}^{-1} w_0  =  w_{-w_0 \alpha}$. Recall that
	\begin{align*}
	n_{w_{\alpha}} = x_{\alpha} (1) x_{-\alpha} (1) x_{\alpha} (1),
	\end{align*}
	for all $\alpha \in \tilde{\Delta}$. In view of \eqref{formula::Involution on Weyl  Basic}, we have
	\begin{align}\label{formula:: Involution on Weyl - One Element}
	\iota_{G,\mathcal{P}} ( n_{w_{\alpha}} )  = x_{-w_0\alpha}(-1) x_{w_0 \alpha} ( -1 ) x_{-w_0\alpha}(-1)  =  n_{w_{-w_0 \alpha}}^{-1}  =  n_{w_{\alpha}^{\star}}^{-1}.
	\end{align}
	If $w = w_{\alpha_1} \cdots w_{\alpha_n}$ is a reduced decomposition of $w \in \widetilde{W}$, then $w^{\star} = w_{\alpha_n}^{\star} \cdots w_{\alpha_1}^{\star}$ is a reduced decomposition of $w^{\star}$. Thus
	\begin{align*}
	 \iota_{G,\mathcal{P}} ( n_w  ) & = \iota_{G,\mathcal{P}} ( n_{w_{\alpha_1}}) \cdots \iota_{G,\mathcal{P}} ( n_{w_{\alpha_n}}) =   ( n_{w_{\alpha_1}^{\star}})^{-1}  \cdots ( n_{w_{\alpha_n}^{\star}})^{-1}  \\
	& = \left(n_{w_{\alpha_n}^{\star}} \cdots n_{w_{\alpha_1}^{\star}} \right)^{-1}
	= (n_{w^{\star}})^{-1}.
	\end{align*}
	\end{proof}

\section{Bessel functions over finite fields}\label{section::Bessel}

The basics of the theory of Bessel functions for representations of algebraic groups over finite fields can be found in \cite{Gelfand-finite}. Our reference for this section is \cite[2.1, 2.2]{Cogdell-PS-Shahidi-Stability}. 

  Assume that $F$ is a finite field. Let $G$ be a connected reductive group over $F$. As any connected reductive group over a finite field is quasi-split, we thus retain all the notations of previous sections. Fix an $F$-pinning of $G$ and a non-trivial additive character of $F$. Let $\psi$ denote the associated nondegenerate character of $U(F)$ as in Section \ref{subsection::Whittaker}. By a theorem \cite[Theorem 49]{Steinberg-Yale-1967} of Steinberg, the representation $\ind_{U(F)}^{G(F)} \psi$ is multiplicity free. The proof of this multiplicity one result is based on Gelfand's trick and the existence of certain anti-involution on $G(F)$, which was constructed for semisimple groups by Steinberg in \emph{loc.cit.}. For general quasi-split reductive groups, the composition of inversion and the duality involution is such an anti-involution, see Remark \ref{rem::Involution=Shalika} and \cite[Theorem 8.1.3]{Carter-1985} (cf. \cite[Appendix]{Shalika-MO}).
  
  Let $(\pi, V_{\pi})$ be an irreducible, $\psi$-generic representation of $G(F)$, that is,
  \begin{align*}
  \dim_{\BC} \textup{Hom}_{G(F)}\left(\pi, \ind_{U(F)}^{G(F)} \psi \right)  =  \dim_{\BC} \textup{Hom}_{U(F)}\left(\pi,  \psi \right)  =  1.
  \end{align*} 
  Take a nonzero $\psi$-Whittaker functional $\Lambda$, that is, $\Lambda (\pi (u) v) = \psi (u)v$ for all $u \in U(F)$ and $v \in V_{\pi}$. As any representation of a finite group is unitary, we can choose a nonzero $v_W \in V_{\pi}$ such that $\pi(u) v_W = \psi (u) v_W$ for all $u \in U(F)$ and $v\in V_{\pi}$. Such $\Lambda$ and $v_W$ are unique up to scalars. Assume further that $\Lambda$ and $v_W$ are normalized so that $\Lambda (v_W) = 1$. 

\begin{defn}
	The (normalized) Bessel function of $\pi$ (with respect to $\psi$) is the function $B_{\pi,\psi}$ on $G(F)$ given by
	\begin{align}\label{formula::defn Bessel}
	B_{\pi, \psi} (g)  = \Lambda ( \pi (g) v_W).
	\end{align}
\end{defn}

  One sees immediately that $B_{\pi,\psi}(1) = 1$ and
  \begin{align}\label{formula::Bessel  Transformation Rule}
  B_{\pi,\psi} (u_1 g u_2)  =  \psi (u_1) \psi (u_2) B_{\pi,\psi} (g)
  \end{align}
  for all $u_1$, $u_2 \in U(F)$. One sees also that $B_{\pi,\psi}$ is nothing but a $\psi$-Whittaker vector in the Whittaker model of $\pi$. By the uniqueness of Whittaker models, the Bessel function $B_{\pi,\psi}$ is uniquely determined by the isomorphism class of $\pi$. On the other hand, $B_{\pi,\psi}$ also determines the isomorphism class of $\pi$ because the restriction of the right regular representation to the space of linear combinations of right translations of $B_{\pi,\psi}$ is isomorphic to $\pi$.

We now take a closer look at the transformation rule $\eqref{formula::Bessel  Transformation Rule}$ satisfied by Bessel functions. We remark that the discussions till the end of this section, except the second statement of Proposition \ref{proposition::Bessel w--support}, are valid for an arbitrary field $F$, and that Proposition \ref{proposition::Bessel w--support}, \ref{proposition:: Bessel Aw}, \ref{proposition::Description of w} are well known to experts and are stated in \cite[2.1, 2.2]{Cogdell-PS-Shahidi-Stability} without proof. We include here the proofs for the sake of completeness.

\begin{defn}
	Let $w \in W$. We say that $w$ is Bessel relevant if for every $a\in \Delta$ such that $wa >0$, $w a$ is simple. We denote by $\mathcal{B}(G)$ the set of all Bessel relevant elements in $W$.
\end{defn}

  \begin{prop}\label{proposition::Bessel w--support}
  Any function on $G(F)$ satisfying \eqref{formula::Bessel  Transformation Rule} is supported on the Bruhat cells $C(w)$ such that $w \in \mathcal{B}(G)$. Moreover, for each $w \in \mathcal{B}(G)$, there is a $\psi$-generic representation of $G(F)$ whose Bessel function is not identically zero when restricting to $C(w)$.
  \end{prop}
   \begin{proof}   	
   	Let $B$ be such a function. If $w \notin \mathcal{B}(G)$, we choose $a \in \Delta$ such that $w a > 0$, $wa \notin \Delta$. As $\psi$ is nondegenerate, we can choose $ u \in U_a(F)$ such that $\psi (u)  \neq 1$. In fact, if $2a \notin \Phi$, $V_a (F) = U_a (F)$; if $2a \in \Phi$, the map $U_a(F) \ra V_a(F)$ is surjective. Let $ n_w $ be the chosen representative of $w \in W$ with respect to the pinning. As $wa \notin \Delta$, we have, for any $t \in T(F)$, $t \shskip n_w u n_w^{-1} t^{-1} \in U_{wa} (F) \subset U_1(F)$. Thus $\psi (t \shskip n_w u\shskip n_w^{-1} t^{-1})  =  1$. Evaluating $B$ at $(t \shskip n_w)u = (t \shskip n_w u \shskip n_w^{-1}t^{-1})( t \shskip  n_w)$, we conclude from \eqref{formula::Bessel  Transformation Rule} that $B ( t n_w) = 0$.
   	
   	As for the second statement, for $w \in \mathcal{B}(G)$, set 
   	\begin{align}\label{formula::Iw}
   	I_w = \{ a \in \Delta \ |\ wa>0  \}.
   	\end{align}
   	Then $w(I_w) \subset \Delta$. Recall that $\Phi^+_w = \{ a \in \Phi^+ \ |\ wa >0 \}$. By simple arguments one can show that $\Phi_w^+$ is exactly the set $[I_w]^+$ of $F$-roots which are both linear combinations of elements in $I_w$ and positive. So $U_{I_w} = U_{\Phi_w^+} = U \cap w^{-1}U w$. By \eqref{formula:: compatible representative character}, we have $\psi (u)  =  \psi ( n_w u  n_w^{-1})$ for all $ u \in U(F) \cap n_w^{-1}U(F) n_w$. This means that we can define a function $f$ on $G(F)$ satisfying \eqref{formula::Bessel  Transformation Rule} such that $f$ is supported only on $U(F) n_w U(F)$. As any function satisfying \eqref{formula::Bessel  Transformation Rule} is a linear combination of Bessel functions of $\psi$-generic representations (this is the only place where we need $F$ to be a finite field), the second statement then follows.
   	\end{proof}

For $w \in \mathcal{B}(G)$, let $\widetilde{I_w} \subset \widetilde{\Delta}$ be the set of absolute simple roots that restrict to an element of $I_w$ (see \eqref{formula::Iw}). Define 

\begin{align*}
 A_w  = \{  t \in T(F) \ |\ w \alpha (t) = 1 \text{ for all $\alpha \in \widetilde{I_w}$} \}.
\end{align*}

%  \begin{align*}
 %  A_w  = \{  t \in T(F) \ |\ w \alpha (t) = 1 \text{ for all $\alpha \in \widetilde{\Delta}_a$ and all $a \in \Delta$ such %that $wa >0$} \}.
%  \end{align*}
%For $a \in \Delta$, let $\widetilde{\Delta}_a$ denote the subset of simple roots in $\Delta$ that restrict to $a$.
  
  \begin{prop}\label{proposition:: Bessel Aw}
  	For any function on $G(F)$ satisfying \eqref{formula::Bessel  Transformation Rule}, its restriction to $C(w)$ is supported on the double cosets $U(F)A_w n_w U(F)$.
  \end{prop}
  \begin{proof}
  	Let $B$ be such a function and $t \in T(F)$. If $ t \notin A_w$, there exists $\alpha \in \widetilde{I_w}$ such that $w\alpha (t) \neq 1$. Note that $w\alpha (t) \in L_{\alpha}$. Thus there exists $y \in L_{\alpha}$ such that $\psi_0 (\shskip \Tr_{L_{\alpha} /F} y)  \neq \psi_0(\shskip \Tr_{L_{\alpha} /F} (w\alpha(t)y))$. Suppose that $\alpha$ restricts to $a \in \Delta$. By Lemma \ref{lemma::representative of Weyl} and \eqref{formula::x_a Case I}, \eqref{formula::x_a Case II}, we have 
   \begin{align*}
     t \shskip n_w \bar{x}_{a} (z) n_w^{-1} t^{-1}  =  \bar{x}_{wa} (w \alpha (t) z), \quad \text{for all }z \in L_{\alpha}.
   \end{align*}
   In view of \eqref{formula::psi-P}, we can choose $u \in U(F)$ such that $\psi (u) \neq \psi (t \shskip n_w u n_w^{-1} t^{-1})$. Therefore $B(t \shskip n_w) = 0$. 
  	\end{proof}

   Bessel relevant Weyl elements and the group $A_w$ have nice descriptions.
   
%    Recall that for $w \in \mathcal{B}(G)$, $I_w = \{ a \in \Delta \ |\ wa >0\}$. Let $\widetilde{I_w} \subset %\widetilde{\Delta}$ be the set of (absolute) simple roots that restrict to an element of $I_w$. 
   
   \begin{prop}\label{proposition::Description of w}
	\textup{(1)}$\ $ The Weyl element $w \in W$ is Bessel relevant if and only if $w = w_0w_I$, where $w_0$ is the longest Weyl element in $W$ and $w_I$ is the longest Weyl element of $W_I = \left\langle w_a \ |\ a \in I \right\rangle$ for some subset $I \subset \Delta$. Moreover, $I = I_w$ is uniquely determined by $w$.
	
	\textup{(2)}$\ $ The group $A_w$ is the $F$-rational points of the center of the Levi subgoup $L_{w\widetilde{I_w}}  =  n_w L_{ \widetilde{I_w}} n_w^{-1}$ that corresponds to $ w \widetilde{I_w}$.
   \end{prop}
    \begin{proof}
    	The statements in $(1)$ is Lemma 89 of \cite{Steinberg-Yale-1967} due to Steinberg. As for $(2)$, by definition, $A_w$ is the $F$-rational points of
    	 \begin{align*}
    	\bigcap_{ \alpha \in \widetilde{I_w}} \ker (w \alpha)  =  n_w \left(\bigcap_{ \alpha \in \widetilde{I_w}} \ker \alpha
    	 \right) n_w^{-1},
    	\end{align*}
    	which is just the center of $L_{w \widetilde{I_w}} = n_w L_{\widetilde{I_w}} n_w^{-1}$.  
    	\end{proof}

  \begin{prop}\label{proposition:: Bessel-  Involution}
  	For $w \in \mathcal{B}(G)$ and $ t  \in A_w$, we have
  	  \begin{align}
  	    \iota_G ( t \shskip n_w ) = ( t \shskip n_w )^{-1}.
  	  \end{align}
  \end{prop}
  \begin{proof}
  	It is equivalent to show that
  	\begin{align}
  	\iota_G ( n_w )^{-1} \iota_G ( t )^{-1} \iota_G( n_w )  \iota_G ( n_w )^{-1} = t \shskip n_w.
  	\end{align}
  	
  	We first show that, if $w$ is Bessel relevant, $\iota_G( n_w ) =  n_w^{-1}$. By Proposition \ref{proposition::Description of w}, we have $w = w_0 w_J$ for some subset $J$ of simple $F$-roots. Thus $w^{\star}  =  w_0 (w_0 w_J)^{-1} w_0  = w$. By Theorem \ref{theorem::Involution}, we then get $\iota_G( n_w ) = n_w^{-1}$.
  	
  	Next we show, for $t \in A_w$, that $\iota_G(t^{-1}) = n_w^{-1} t \shskip  n_w$. Write simply $I = I_w$. As $w_0 = w w_I$ and $l (w_0)  =  l(w) + l(w_I)$, we have $ n_{w_0}  =  n_w  n_{w_I}$. By Proposition \ref{proposition::Description of w}, $ n_w^{-1} t \shskip n_w$ is in the center of $L_{\widetilde{I_w}}$, and thus commutes with $ n_{w_{I}}$. Therefore $ n_w^{-1} t \shskip  n_w = n_{w_I}^{-1} n_w^{-1} t \shskip n_w  n_{w_I} =  n_{w_0}^{-1} t \shskip n_{w_0} = n_{w_0} t\shskip   n_{w_0}^{-1} =\iota_G(t^{-1})$.
  	
  	\end{proof}
 
    \begin{cor}\label{cor::Involution-Bessel}
    	For any function $f$ on $G(F)$ satisfying \eqref{formula::Bessel  Transformation Rule}, if $f(n) \neq 0$ with $n \in N_G(T)(F)$, then $\iota_G (n) = n^{-1}$.
    \end{cor}

\begin{rem}\label{rem::Involution=Shalika}
	We now verify that the composition of the duality involution with the inversion, $g \mapsto \iota_{G,\mathcal{P}}(g^{-1})$, satisfies the properties $(1.1)$-$(1.5)$ listed on page 174 of Shalika's multiplicity one paper \cite{Shalika-MO}. Properties $(1.1)$-$(1.3)$ are immediate from the definition. By \eqref{formula::algebraic Whittaker} and \eqref{formula:: Involution on Weyl - One Element}, we have $\psi (\iota_{G,\mathcal{P}} (u))  =  \psi (u)^{-1}$, hence property $(1.4)$ holds. As for $(1.5)$, fix $n \in N_G(T)(F)$. If $\psi (u) = \psi (n u n^{-1})$, for all $u \in U \cap n^{-1} U n$, then we can define a function $f$ on $G(F)$ that is supported on $U(F)nU(F)$, thus $\iota_{G,\mathcal{P}} (n^{-1}) = n$ by Corollary \ref{cor::Involution-Bessel}. 
\end{rem}

\section{The opposition group and Shintani base change}\label{section::opposition Shintani}

Let $E / F$ be a quadratic extension of finite fields. Let $G$ be a connected reductive group over $F$, hence quasi-split, with a fixed $F$-pinning $\mathcal{P}$. To ease the notations, we will omit the pinning $\mathcal{P}$ if this causes no confusion. 

Recall that the opposition group $G^{op}$ is defined to be the quasi-split group obtained by twisting $G$ by the cocycle
\begin{align*}
  z \colon \Gal(\bar{F} / F)    \ra \text{Aut}(G)
\end{align*}
which sends the non-trivial element in $\Gal(E/F)$ to the duality involution $\iota_{G}$ (see \cite{Prasad-Relative+LLC}). Thus $G^{op}$ is a $E$-form of $G$. 

 Let $\sigma_0$ denote the Frobenius endomorphism associated to the $F$-rational structure of $G$.  So $G(F) = G^{\sigma_0} $. Let $ \sigma = \sigma_0 \circ \iota_G$. Then $\sigma$ is the Frobenius map associated to the $F$-rational structure of $G^{op}$. We have
  
  	\begin{align*}
  	G^{op}(F) & = G^{\sigma} = \{ g \in G(E) \ |\ \sigma (g) = g  \}, \\
  	G^{op}(E) & = G^{\sigma^2} = G(E).
  	\end{align*}
  	
%% \footnote{We warn the reader that there is a slight difference between the definition one here and the one in %%\cite{Prasad-Relative+LLC}.}

\subsection{Twisted action}

Given $g$, $x \in G^{\sigma^2}$, the twisted action associated to $\sigma$ is given by $g \cdot x  =  g x \sigma(g)^{-1}$. The orbits of the twisted action are called $\sigma$-conjugacy classes of $G^{\sigma^2}$. Let 
\begin{align}\label{formula::X-sigma}
X_{\sigma} = \{ g \in G^{\sigma^2} \ |\  \sigma (g) = g^{-1}   \}.
\end{align} 
Note that $g \mapsto g \cdot 1$ gives an injection from $G^{\sigma^2} / G^{\sigma}$ to $X_{\sigma}$. By Lang-Steinberg's theorem, every element $g$ of $G$ can be written as $h \sigma(h)^{-1}$; the choice of $h$ is unique up to right translation by $G^{\sigma}$. One sees that $g \in X_{\sigma}$ if and only if $ h \in G^{\sigma^2}$. Thus $X_{\sigma}$ is exactly the $\sigma$-conjugacy class of $1$ in $G^{\sigma^2}$, and we have a bijection map $G^{\sigma^2}  /  G^{\sigma}  \ra X_{\sigma}$.

By the definition of $\iota_{G}$, the groups $B$, $T$ and $U$ are stable under the action of $\sigma$. The following proposition, which is a special case of \cite[Proposition 6.6]{Helminck-Wang-Involution}, gives a complete set of representatives for the twisted action of $U^{\sigma^2} = U(E)$ on $X_{\sigma}$. Since the formulation is slightly different, we are prompted to include the elegant proof in \emph{loc.cit.} here. 

    \begin{prop}\label{lemma::Norm one elements}
    	If $g \in X_{\sigma}$, then $g = u  n  \sigma(u)^{-1}$ for some $u \in U(E)$ and $n \in N_G(T)(E)$ with $\sigma (n)  =  n^{-1}$. 
    \end{prop}
     \begin{proof}
     	 By the $E$-rational Bruhat decomposition, we write $g = u_1 n \sigma(u_2)$, with $u_1$, $u_2 \in U(E)$ and $n \in N_G(T)(E)$. We may assume that $n^{-1}u_1 n \in U^-(E)$. By the assumption we have 
     	\begin{align}\label{formula::prop 5.5-Fg=g-1}
     	\sigma(u_1) \sigma(n) u_2  =  \sigma(u_2)^{-1} n^{-1} u_1^{-1}.
     	\end{align} 
     	This implies that $\sigma(n)  =  n^{-1}$. Write $u_2^{-1}  =  v_1v_2$ such that
     	\begin{align*}
     	v_1 \in U(E) \cap n U^-(E) n^{-1} \quad \text{and}\quad v_2 \in U(E) \cap n U(E) n^{-1}.
     	\end{align*}
     	Rewrite the equality \eqref{formula::prop 5.5-Fg=g-1} as
     	\begin{align*}
     	(\sigma(u_1) n^{-1} v_2^{-1} n ) n^{-1} v_1^{-1} n = ( \sigma(v_1) \sigma(v_2)) n^{-1} u_1^{-1} n.
       	\end{align*}
       	By the disjointness of $U$ and $U^{-}$, we have $u_1 = v_1$ and $\sigma(v_2)  =  n^{-1} v_2^{-1} n$. Hence we have 
       	\begin{align}
       	v_1^{-1} g \sigma(v_1) = v_2 n.
       	\end{align}
       	
     	Now consider the Frobenius endomorphism $\sigma^{\prime} = \mathrm{Ad}(n)\circ \sigma$. The fact that $ \sigma(n) = n^{-1} $ implies that $\sigma^{\prime}$ takes $ U \cap n U n^{-1}$ into itself. So, by Lang-Steinberg's theorem, we can choose $y \in U \cap nUn^{-1}$ such that $y n \sigma(y^{-1})n^{-1} = v_2$. As $v_2n\in X_{\sigma}$, one checks that $y \in U(E)$. Therefore the element $u = v_1 y \in U(E)$ has the desired property.
      	\end{proof}

The next lemma will be used in Proposition \ref{proposition::Counting}. Its proof is a simple application of Lang-Steinberg's theorem, so we omit it here (see, for example, \cite[\S 1.17]{Carter-1985}).
\begin{lem}\label{lemma::Quotient}
	Let $H$ be a $\sigma$-stable closed subgroup of $G$. Assume that $H$ is connected and is normal in $G$. Then we have a bijection
	\begin{align}\label{formula::Quotient}
	G^{\sigma} / H^{\sigma}  \cong \left( G^{\sigma^2}  /  H^{\sigma^2} \right)^{\sigma}.
	\end{align}
\end{lem}

\subsection{Shintani base change}\label{section::Shintani}

We now recall the notion of Shintani base change in the representation theory of finite reductive groups. 

  Recall that $g_1$, $g_2 \in G^{\sigma^2}$ are $\sigma$-conjugate if $g_1 = x g_2 \sigma(x)^{-1}$ for some $x \in G^{\sigma^2}$. We denote by $C_{G^{\sigma^2},\shskip \sigma}(g)$ the $\sigma$-conjugacy class of $g$ in $G^{\sigma^2}$. By 
  Lang-Steinberg's theorem, every element $g$ of $G$ can be written as $h \sigma(h)^{-1}$; the choice of $h$ is unique up to right translation by $G^{\sigma}$. It can be checked easily that $g \in G^{\sigma^2}$ if and only if 
  \begin{align*}
  \sigma^{2}(h)^{-1} h = \sigma(h)^{-1} (\sigma(g) g) \sigma(h) \in G^{\sigma}.
  \end{align*}
  Thus we have a well-defined map 
  \begin{align}\label{formula::Norm map N}
  \EuScript N \colon  \{\sigma\text{-conjugacy classes of } G^{\sigma^2}  \} & \ra  \{\text{conjugacy classes of }G^{\sigma}  \}    \\
   C_{G^{\sigma^2}, \sigma} (g)  & \mapsto   C_{G^{\sigma}} (\sigma^2(h)^{-1}h). \nonumber
  \end{align}
  The map $\EuScript N$ is called the Shintani norm map.
  
\begin{lem}
	The Shintani norm map $\EuScript N$ is bijective.
\end{lem}
\begin{proof}
	This is proved in \cite[Lemma 2.2]{Kawanaka-Unitary+ShintaniLift}
	\end{proof}

Let $A = \langle 1, \sigma^{\prime} \rangle$ be the group of order $2$. Suppose that $A$ acts on $G(E)$ by $\sigma^{\prime}(g)  =  \sigma(g)$. We shall write $\sigma$ for $\sigma^{\prime}$. Denote by $\tilde{G}(E)$ the semi-direct product $G(E) \rtimes A$. The component $G(E) \rtimes \sigma$ of $\tilde{G}(E)$ is stable under the conjugation action of $\tilde{G}(E)$, and the conjugacy classes in $G(E) \rtimes \sigma$ are in bijection with $\sigma$-conjugacy classes of $G(E)$. 

We say that a representation $\pi$ of $G(E)$ is $\sigma$-invariant if it is equivalent to the representation $\pi^{\sigma}$ defined by
 \begin{align*}
   \pi^{\sigma} (g)  =  \pi (\sigma (g)), \quad  g \in G(E).
 \end{align*}
It is shown by Kawanaka \cite[Corollary 2.3]{Kawanaka-Unitary+ShintaniLift} that the number of irreducible, $\sigma$-invariant representations of $G(E)$ is equal to the number of irreducible representations of $G^{op}(F)$. For an irreducible, $\sigma$-invariant representation $\pi$ of $G(E)$, by Schur's lemma, we can extend $\pi$ to a representation $\tilde{\pi}$ of $\tilde{G}(E)$. This amounts to a choice of intertwining operator $I_{\sigma}$ between $\pi$ and $\pi^{\sigma}$ such that $I_{\sigma}^2  =  1$. Clearly there are two choices of such $I_{\sigma}$.

\begin{defn}\label{defn::Shantani bc}
	Let $\pi$, $\rho$ be irreducible representations of $G(E)$, $G^{op}(F)$ respectively. Assume that $\pi$ is $\sigma$-invariant. We say that $\pi$ is a (Shintani) base change lift of $\rho$ if, for an extension $\tilde{\pi}$ of $\pi$ to $\tilde{G}(E)$, 
	\begin{align}\label{formula::Shintani bc-defn}
	  \chiup_{\tilde{\pi}} ( g \rtimes \sigma)  =  \varepsilon (\tilde{\pi}) \chiup_{\rho} ( \EuScript N (g)), \quad g \in G(E),
	\end{align}
	where $\chiup$ is the trace character of a representation, $\varepsilon (\tilde{\pi})  =  \pm 1$, and $\EuScript N(g)$, by abuse of notation, is the image of the $\sigma$-conjugacy class of $g$ under the Shintani norm map $\EuScript N$ \eqref{formula::Norm map N}.
\end{defn}

 \begin{rem}\label{rem::Shintani bc}
 	 For a generic representation $\pi$, we can make a canonical choice of the intertwining operator $I_{\sigma}$, following the treatment in \cite{Arthur-Clozel}. This will rely on Whittaker models. Assume that $\pi$ is $\sigma$-invariant and $\psi$-generic, for a nondegenerate character $\psi$ of $U(E)$. This forces that $\psi = \psi^{\sigma}$, which we shall also assume. Let $\EuScript V  \subset \ind_{U(E)}^{G(E)} \psi$ be the unique subspace of Whittaker models of $\pi \cong \pi^{\sigma}$. For $f \in \EuScript V$, let $f^{\sigma}(g) = f( \sigma(g))$. Now we define $I_{\sigma} \colon \EuScript V \ra \EuScript V$ by $I_{\sigma} (f)  =  f^{\sigma}$. Also, we shall write $\varepsilon(\pi)$ for $\varepsilon(\tilde{\pi})$, with $\tilde{\pi}$ the extension corresponding to this normalized intertwining operator $I_{\sigma}$.
 	
 \end{rem}

  \begin{rem}
  The definition of the norm map \eqref{formula::Norm map N} here is due to Kawanaka \cite[\S 2]{Kawanaka-Unitary+ShintaniLift}. In \cite{Shintani}, Shintani considered the norm map $N$ defined by $N (g) = g \sigma (g) $. For $G = \GL_n$ and $\sigma = \sigma_0$, he proves that the conjugacy class of $N(g)$ in $G$ contains exactly one conjugacy class of $G^{\sigma}$. Thus he obtains the same norm map between conjugacy classes as $\EuScript N$ \eqref{formula::Norm map N}. A general result of Springer and Steinberg \cite[1, 3.4]{Springer-Steinberg-ConjugacyClass} says that if the centralizer $Z_G (x)$ is connected for all $x \in G$ , the conjugacy class of $N(g)$ in $G$ contains exactly one conjugacy class of $G^{\sigma}$. 
 \end{rem}

\section{Proof of Theorem \ref{theorem::Main}}\label{section::Main}

 Let $F$ be a finite field of $q$ elements and $E / F$ be a quadratic field extension. Let $G$ be a connected reductive group over $F$ with an $F$-pinning $\mathcal{P}$. Let $\psi_0$ be a non-trivial additive character of $ E / F$, and $\psi$ be the nondegenerate character of $U(E)$ that associated to $\psi_0$ and the pinning $\mathcal{P}$ via \eqref{formula::algebraic Whittaker} (here $G$ and $\mathcal{P}$ are viewed as defined over $E$). We have also the duality involution $\iota_G$ and the opposition group $G^{op}$ associated to $\mathcal{P}$. 

  Let $\pi$ be an irreducible, $\psi$-generic representation of $G(E)$. Consider the obvious $G(F)$-invariant linear form $\lambda$ on the Whittaker model of $\pi$:
  
  \begin{align}\label{formula::Distinction lin-form}
  \lambda (W) = \sum_{h \in G(F)} W (h),\quad W \in \EuScript W(\pi,\psi).
  \end{align}
  We will show that, if $\pi$ is a base change lift of some representation $\rho$ of $G^{op}(F)$, the $G(F)$-invariant functional $\lambda$ is not zero. The key idea of the method in \cite{Anan***-Matringe-BaseChange} is to compare $\lambda$ with another linear functional $\mu$ on the Whittaker model:
  
  \begin{align}\label{formula::Whittaker lin-form}
  \mu (W)  = \sum_{ g \in X_{\sigma} }  W ( g ), \quad W \in \EuScript W(\pi,\psi),
  \end{align}
  where $X_{\sigma}$ is defined in \eqref{formula::X-sigma}. 

First, one observes that

\begin{lem}\label{lemma:: generic charcter}
    The nondegenerate character $\psi$ is trivial on $U(F)$; $\psi (u) =  \psi^{\sigma} (u) = \psi (\sigma (u))$ for all $u \in U(E)$.
\end{lem}
\begin{proof}
	The first statement follows from the fact that the additive character $\psi_0$ is trivial on $F$.  For the second statement, it suffices to verify that $\psi = \psi^{\sigma}$ on $V_a(E)$ for every $a \in \Delta$. Identifying $V_a(E)$ with $L_{\alpha} \otimes_{F} E$ via \eqref{formula::Va}, we have $\psi (z) = \psi_0 \left( \Tr_{E} (z) \right) $ for all $z \in L_{\alpha} \otimes_{F} E$, where $\Tr_{E} \colon L_{\alpha} \otimes_{F} E \ra E$ is defined by tensoring the trace map $\Tr \colon L_{\alpha} \ra F$ with $E$. Note that $\sigma$ takes $V_a(E)$ into $V_{-w_0 a}(E)$, which can also be identified with $L_{\alpha} \otimes_{F} E$ via \eqref{formula::Va}. In view of \eqref{formula::Involution on Weyl  Basic}, we can transport $\sigma$ to an action on $L_{\alpha} \otimes_F E$, still denoted by $\sigma$, and get that
	\begin{align}
	  \sigma \left( \sum u_i \otimes v_i \right)  =  - \sum (u_i \otimes v_i^q).
	\end{align}
	Thus $\Tr_{E} ( \sigma(z) )  =  - \Tr_{E}(z)^q$ for all $z \in L_{\alpha} \otimes_{F} E$. The statement then follows from the fact that $\psi_0$ is trivial on $F$.
	\end{proof}
\begin{rem}
	Conversely, one can check that any nondegenerate character of $U(E)$ that is trivial on $U(F)$ can be obtained in this way. In view of Lemma \ref{lemma:: generic charcter}, one can check easily that $\mu$ is actually a $\psi$-Whittaker linear form (see \cite[Remark 7]{Anan***-Matringe-BaseChange}).
\end{rem}

\begin{prop}\label{proposition:: lambda  =  mu}
	Let $B_{\pi,\psi}$ be the $\psi$-Bessel function associated to $\pi$. Then we have $\lambda (B_{\pi,\psi}) = \mu ( B_{\pi,\psi} )$.
\end{prop}
 \begin{proof}
 	There are actions of $U(F) \times U( F)$ and $U(E)$ on $G(F)$ and $X_{\sigma}$ respectively, given by
 	\begin{align}
 	(u_1,u_2)\cdot g & =  u_1 g u_2^{-1},      \quad  u_1,u_2 \in U(F),\  g \in G(F) , \\
      	u \cdot g    & =  \sigma (u)  g u^{-1},\quad nu \in U(E), \ g \in X_{\iota}. \label{formula::Twist Action}
 	\end{align}
 	By the transformation rule \eqref{formula::Bessel  Transformation Rule} of Bessel functions and Lemma \ref{lemma:: generic charcter}, $B_{\pi,\psi}$ takes the same value on every orbit of the two actions. 
 	
 	By the $F$-rational Bruhat decomposition for $G(F)$ and Proposition \ref{lemma::Norm one elements}, we can take $N_G(T)(F)$ and $N_G(T)(E) \cap X_{\sigma}$ respectively as a complete set of orbit representatives for the two actions.
 		
 	Set 
 	\begin{align*}
 	X_1 = & \{ n \in N_G(T)(F)\ |\ B_{\pi,\psi}( n ) \neq 0) \}; \\
 	X_2 = & \{ n \in N_G(T)(E)  \cap X_{\sigma}\ |\ B_{\pi,\psi} ( n ) \neq 0 \}.
 	\end{align*}
  By Corollary \ref{cor::Involution-Bessel}, we have $\iota_G( n ) =  n^{-1}$ for those $n\in N_G(T)(E)$ with $B_{\pi,\psi} (n) \neq 0$. This implies that the two sets $X_1$ and $X_2$ are equal.
 	 	
    To prove $\lambda(B_{\pi,\psi})  =  \mu(B_{\pi,\psi})$, it suffices to prove that, for each $n \in X_1 = X_2$, the orbits of $n $ under the two actions have the same number of elements. The proposition will then be completed after Proposition \ref{proposition::Counting} where we investigate the stabilizers of $n$ under the two actions.
%In fact, we only need to show that the \red{stabilizer}s in each case have the same number of elements as $\lvert U(F) \times %U(F) \rvert = \lvert U(E) \rvert = q^{2 |\Phi^+|}$. For $U(F) \times U(F)$ acting on $G(F)$, the map $(n_1,n_2) %\ra n_2$ is an isomorphism
%    \begin{align*}
%    \textup{Stab}_{U(F) \times U(F)}(a\dot{w}) \cong U(F) \cap \dot{w}^{-1} U(F) \dot{w}.
%    \end{align*}
%    In view of \eqref{formula::Bruhat decompostion }, we have $\lvert U(F) \cap \dot{w}^{-1} U(F) \dot{w} \rvert = q^{ %\lvert \Phi^+_{w}\rvert } = q^{ |\Phi^+| - l(w)}$. The proposition will be completed after Proposition %\ref{proposition::\red{stabilizer}} where we investigate in detail the \red{stabilizer}s of the remaining case.
  \end{proof}

\begin{prop}\label{proposition::Counting}
	Write $n = a n_w$ with $w$ Bessel relevant in $W$. Then
	\begin{align*}
	\left\lvert U(F) \times U(F) \right\rvert  & = \left\lvert U(E) \right\rvert = q^{2 |\widetilde{\Phi}^+|},  \\
	\left\lvert \mathrm{Stab}_{U(F) \times U(F)}(n) \right\rvert  &= \left\lvert \mathrm{Stab}_{U(E)} (n)  \right\rvert  = q^{ |\widetilde{\Phi}^+| - \tilde{l}(w)},
	\end{align*}
	where $\tilde{l}(w)$ is the length of $w$ viewed as an element in $\widetilde{W}$.
%	Suppose that $a \dot{w} \in X_{\iota}$. The action of $U(E)$ on $X_{\tau}$ is given in \eqref{formula::Twist Action}. %Then
%	\begin{align*}
%	\lvert \textup{Stab}_{U(E)} (a \dot{w})  \rvert  = q^{ |\Phi^+| - l(w)}.
%	\end{align*}
\end{prop}
  \begin{proof}
  	We first recall some well-known facts that will be used in our argument (see \cite[\S 1.3]{Bushnell-Henniart-DerivedSubgroup}). Let $\phi$ be a closed subset of $\Phi^+$ and $U_{\phi}$ be the corresponding $F$-unipotent subgroup. For $a \in \Phi^+$, let $l_{\Delta} (a)$ be the number of simple roots, counting multiplicities, that appeared in the decomposition of $a$ as a sum of simple roots. So $l_{\Delta} (a) = 1$ if and only if $a \in \Delta$. Set $\phi_i = \{a \in \phi \ |\ l_{\Delta} (a) > i \}$. So we have a decreasing sequence of closed subsets of $\phi$,
  	\begin{align*}
  	\phi = \phi_0 \supseteq \phi_1 \supseteq \phi_2 \supseteq \cdots
  	\end{align*}
  	Let $U_{\phi_i}$ be the $F$-unipoptent subgroup corresponding to $\phi_i$. Then $U_{\phi_i}$ is a normal subgroup of $U_{\phi_{i-1}}$ and we have an $F$-isomorphism
    \begin{align}\label{formula::Counting-Quotient}
    \prod_{a\in \phi,\  l_{\Delta} (a) = i} V_a \cong U_{\phi_{i-1}} / U_{\phi_i}.
    \end{align}
%  	Thus if we denote by $\tilde{\phi} \subset \widetilde{\Phi}^+$ the subset of absolute roots that restrict to some %element in $\phi$, we have $ \lvert U_{\phi}(F) \rvert  =  q^{\lvert \tilde{\phi} \rvert}$.

  	 When $\phi = \Phi^+$, the discussion above, together with \eqref{formula::Va}, implies that 
  	 \begin{align}
  	   \lvert U(F) \rvert = q^{ |\shskip \widetilde{\Phi}^+|} \quad \text{and}\quad  \vert U(E)\vert = q^{2 |\shskip \widetilde{\Phi}^+ |}.
  	 \end{align}
  	 
  	 For the action of $U(F) \times U(F)$ on $G(F)$, the map $(n_1,n_2) \mapsto n_2$ is an isomorphism
  	    \begin{align*}
  	    \textup{Stab}_{U(F) \times U(F)}( n ) \cong U(F) \cap n_w^{-1} U(F) n_w.
  	    \end{align*}
  	Hence 
  	\begin{align}
  	  \lvert \mathrm{Stab}_{U(F) \times U(F)}(n) \rvert  =  \vert U_{\Phi^+_w} (F) \vert  =  q^{\vert \widetilde{\Phi}^+\vert  - \tilde{l}(w)}.
  	\end{align}
  	
  	For the action of $U(E)$ on $G(E)$, we see from \eqref{formula::Twist Action} that $u$ lies in the stabilizer of $n$ if and only if $ u = n^{-1} \sigma(u) n$, which implies that $u \in U_{\Phi_w^+} (E)$. Let $\tau$ denote the Frobenius endomorphism $u \ra n^{-1} \sigma (u ) n$ on $U_{\Phi_w^+}$. Note that $\tau$ induces an involution on $U_{\Phi_w^+} (E)$ which we still denote by $\tau$. Thus the stabilizer of $n$ in $U(E)$ is the fixed points of $\tau$ in $U_{\Phi_w^+} (E)$. 
  	
  	Let $I = I_w$ be the subset of simple roots associated to $w$ as in \eqref{formula::Iw}. Recall that $\Phi_w^+ =  [I_w]^+$, the positive roots that are linear combinations of elements in $I_w$. Let $\phi = [I_w]^+$ and $\phi_i$ be as above. By Proposition \ref{proposition::Description of w}, we can write $w = w_0 w_{I}$ with $w_I$ the unique long Weyl element associated to $I \subset \Delta$. Since $ - w_I (I) = I$, we have $l_{\Delta} (a)  =  l_{\Delta} (-w_I a)$ for all $a \in [I_w]^+$. Note that
  	\begin{align*}
  	  \tau ( U_a) \subset U_{ - w_I a}, \quad a \in [I_w]^+.
  	\end{align*}
  	Thus each $U_{\phi_i}$ is $\tau$-stable. By Lemma \ref{lemma::Quotient}, we then have
  	\begin{align*}
      \vert \mathrm{Stab}_{U(E)} (n)  \vert =  \prod \  \vert  (U_{\phi_{i-1}}(E) / U_{\phi_i}(E))^{\tau}  \vert .
  	\end{align*}
  	
  	By \eqref{formula::Counting-Quotient}, we now transport the action of $\tau$ to $\prod_{a\in \phi, l_{\Delta}(a) = i} V_a (E)$. For $a$ such that $-w_I a = a$, the action of $\tau$ on $V_a(E)$, which is identified with $V_a(F) \otimes_{F} E$ by \eqref{formula::Va}, is given by an $F$-linear map on $V_a(F)$ of order $2$ tensored with the nontrivial element in the Galois group $\Gal(E / F)$. So we get that $ \vert V_a (E)^{\tau} \vert  =  \vert V_a (F) \vert$. Therefore, by some elementary arguments, we can conclude easily that
  	\begin{align}
  	  \left\lvert \mathrm{Stab}_{U(E)} (n)  \right\rvert  =  \vert U_{[I_w]^+} (F)\vert  =  q^{\vert \widetilde{\Phi}^+ \vert - \tilde{l} (w)}.
  	\end{align}
  	\end{proof}

It turns out that it is easier to evaluate $\mu(B_{\pi,\psi})$. The arguments in \cite[Lemma 3.4]{Anan***-Matringe-BaseChange} works verbatim here. We include the proof here for the sake of completeness.

\begin{prop}\label{proposition:: Evaluaiton Mu}
	Let $\pi$ be an irreducible, $\psi$-generic representation of $G(E)$ which is a base change lift of a representation $\rho$ of $G^{op}(F)$. Then
	\begin{align}
	  \sum_{ g \in X_{\sigma} } B_{\pi,\psi} (g)  =  \varepsilon(\pi) \shskip \frac{\dim \rho}{\dim \pi} \shskip |X_{\sigma}|,
	\end{align}
	where $\varepsilon(\pi) = \pm 1$ is the quantity defined in Definition \ref{defn::Shantani bc} and Remark \ref{rem::Shintani bc}, and $\dim$ is the dimension of a representation (the prior condition $\psi = \psi^{\sigma}$ in Remark \ref{rem::Shintani bc} is guaranteed by Lemma \ref{lemma:: generic charcter}).
\end{prop}
  \begin{proof}
  	Consider the operator 
  	\begin{align}\label{formula::T operator}
  	T = \sum_{ g \in X_{\sigma} } \pi (g) I_{\sigma}
  	\end{align}
  	defined on the Whittaker model $\EuScript W(\pi,\psi)$ of $\pi$, where $I_{\sigma}$ is defined in Remark \ref{rem::Shintani bc}. Then $T$ intertwines $\pi$ with itself. Indeed, for $x \in G(E)$,
  	\begin{align*}
  	T \pi (x) & = \sum_{ g \in X_{\sigma} } \pi (g) I_{\sigma} \pi (x) 
  	            = \sum_{ g \in X_{\sigma} } \pi (g \sigma(x))  I_{\sigma} \\
  	          & = \sum_{ g \in X_{\sigma} } \pi ( x x^{-1}g \sigma(x)) I_{\sigma} 
  	            = \sum_{ g \in X_{\sigma} } \pi (x) \pi (g) I_{\sigma}  = \pi (x) T.
  	\end{align*}
  	So $T$ is of the form $c(\pi)\cdot I$ for some $c(\pi) \in \BC^{\times}$. 
  	
  	If $\pi$ is a base change lift of a representation $\rho$ of $G^{op}(F)$, by \eqref{formula::Shintani bc-defn}, we have
  	\begin{align}
  	\textup{trace}\left( \pi(g) I_{\sigma} \right)  =  \varepsilon(\pi) \cdot \textup{trace }\rho (\EuScript N(g)), \quad g \in G(E).
  	\end{align}
  	Note that $X_{\sigma}$ is the $\sigma$-conjugacy class of $1$ in $G(E)$, hence its image under the Shintani norm map $\EuScript N$ is the conjugacy class of $1$ in $G^{op}(F)$. Taking traces on both sides of \eqref{formula::T operator}, we get
  	\begin{align}\label{formula:: Mu }
  	c(\pi) \cdot \dim \pi  = \varepsilon(\pi)  \dim \rho \cdot \lvert X_{\sigma} \rvert.
  	\end{align}
  	
  	Note that $ I_{\sigma}(B_{\pi,\psi}) = B_{\pi,\psi}$ by the uniqueness of the $\psi$-Bessel function. Thus
  	\begin{align*}
  	\sum_{ g \in X_{\sigma} } B_{\pi,\psi} (g) & = \sum_{ g \in X_{\sigma} } (\pi(g) I_{\sigma}B_{\pi,\psi}) (I)  = (TB_{\pi,\psi}) (I) 
  	            \\
  	          & = c(\pi)B_{\pi,\psi}(I) = c(\pi).
  	\end{align*}
   	\end{proof}

\begin{thm}
	Let the notations be as above. Let $\pi$ be an irreducible, $\psi$-generic representation of $G(E)$. Assume that $\pi$ is a base change lift of a representation $\rho$ of $G^{op}(F)$. Let $\lambda$ be the $G(F)$-invariant linear functional on the Whittaker model $\EuScript W(\pi,\psi)$ of $\pi$ defined by 
	\begin{align*}
	  \lambda (W)  =  \sum_{h \in G(F)}  W (h),
	\end{align*}
	then
	 \begin{align}
	   \lambda (B_{\pi,\psi})  =\varepsilon(\pi) \frac{\dim \rho}{\dim \pi} \frac{|G(E)|}{|G^{op}(F)|},
	 \end{align}
	where $B_{\pi,\psi}$ is the $\psi$-Bessel function of $\pi$ and $\varepsilon(\pi) = \pm 1$ is the quantity defined in Definition \ref{defn::Shantani bc} and Remark \ref{rem::Shintani bc}. In particular, $\pi$ is distinguished by $G(F)$.
	
\end{thm}
 \begin{proof}
 	The theorem follows from Proposition \ref{proposition:: lambda  =  mu}, \ref{proposition:: Evaluaiton Mu} and the fact that $G(E ) / G^{op}(F)  \cong X_{\sigma}$.
 	\end{proof}

\section{Contragredient and the duality involution $\iota_{G,\mathcal{P}}$}\label{section::remark}

\subsection{Contragredient in the finite field case}

Let $E/F$ be a quadratic extension of finite fields. In \cite{Prasad-Quadratic-Compositio}, Prasad proved that, for a connected reductive group $G$ over $F$, an irreducible representation $\pi$ of $G(E)$ which is a virtual sum of Deligne-Lusztig representations is distinguished by $G(F)$ if and only if 
\begin{align}\label{formula::Condition-Prasad}
 \pi^{\sigma_0} \cong \pi^{\vee},
\end{align}
where $\sigma_0$ is the Frobenius map associated to the $F$-structure of $G$ and $\pi^{\vee}$ denotes the contragredient representation of $\pi$.

 Recall that a prior condition for $\pi$ being a base change lift of some representation of $G^{op}(F)$ is that
\begin{align}\label{formula::Condition-Shintani}
 \pi \cong \pi^{\sigma},
\end{align}
with $\sigma = \sigma_0 \circ \iota_{G}$ the Frobenius map associated to the $F$-structure of $G^{op}$.

We show that, for generic representations, conditions \eqref{formula::Condition-Prasad} and \eqref{formula::Condition-Shintani} are equivalent.

\begin{prop}\label{prop::Contragredient finite}
	%Let $G$ be a connected reductive group over a finite field $F$. 
	Let $\pi$ be an irreducible representation of $G(F)$ which is generic with respect to a nondegenerate character $\psi_{\mathcal{P}}$ associated to an $F$-pinning $\mathcal{P}$. Let $\iota_{G,\mathcal{P}}$ be the associated duality involution. Then
	\begin{align}
	\pi^{\iota_{G,\mathcal{P}}}   \cong   \pi^{\vee}.
	\end{align}
\end{prop}
\begin{proof}
	We show that the two representations $\pi^{\iota_{G,\mathcal{P}}}$ and $\pi^{\vee}$ have a common matrix coefficient. We shall write $\psi$ for $\psi_{\mathcal{P}}$ and $\iota_{G}$ for $\iota_{G,\mathcal{P}}$. Note that $\pi^{\iota_{G}}$ and $\pi^{\vee}$ are all generic with respect to $\psi^{-1}$. Let $\langle \cdot , \cdot \rangle$ be a $G(F)$-invariant inner product on $\pi$, and $v$ be the $\psi$-Whittaker vector in $\pi$ such that $\langle v, v \rangle  =  1$. By the definition of Bessel functions, we see that the $\psi^{-1}$-Bessel functions for $\pi^{\iota_{G}}$ and $\pi^{\vee}$ are
	\begin{align*}
	B_{\pi^{\iota_{G}},\psi^{-1}} (g) & =  \langle \pi^{\iota_{G}}(g) v, v \rangle = B_{\pi,\psi} (\iota_{G} (g)) \\
	B_{\pi^{\vee},\psi^{-1}}  (g)               & = \langle v, \pi(g) v \rangle =  B_{\pi,\psi} (g^{-1}). 
	\end{align*}
	By Corollary \ref{cor::Involution-Bessel}, we get that $B_{\pi^{\iota_{G}},\psi^{-1}}(g)  =  B_{\pi^{\vee},\psi^{-1 } }(g) $ for all $g \in G(F)$. As Bessel functions are matrix coefficients, $\pi^{\iota_{G}}$ and $\pi^{\vee}$ are isomorphic. 
\end{proof}

\subsection{Contragredient in $p$-adic case}

We consider the $p$-adic version of Proposition \ref{prop::Contragredient finite} in this section, which is conjectured in \cite[Conjecture 1]{Prasad-Involution}. Let $F$ be a $p$-adic field. We shall write $G$ for $G(F)$.

\begin{defn}
Let $(\pi,V)$ be an irreducible admissible representation of $G$ and $(\pi^{\vee},V^{\vee})$ be its contragredient. Let 
\begin{align*}
l  \colon  & V \ra \BC;  \\
l' \colon  & V^{\vee}  \ra \BC,
\end{align*}
be nonzero linear forms on $\pi$ and $\pi^{\vee}$. The relative character $B_{\pi,l,l'}$ is a distribution on $G$ defined by 
\begin{align}
B_{\pi,l,l'} (f)  =  l' (\pi (f) l),
\end{align}
for all $f \in C_c^{\infty} (G)$. 
\end{defn}

It is well known that the relative character of a representation characterizes the representation. See \cite[Proposition 3]{Prasad-SelfDual} for the proof of the following lemma.
\begin{lem}\label{lem::relative charac-}
Let $\pi_1$ and $\pi_2$ be two irreducible admissible representations of $G$. If a relative character of $\pi_1$ is equal to a relative character of $\pi_2$, then $\pi_1$ is isomorphic to $\pi_2$.
\end{lem}

\begin{lem}\label{lem::2}
Let $H$ be a subgroup of $G$ and $\chiup$\shskip \shskip, $\chiup^{\vee}$ be two characters of $H$. Let $l$ and $l'$ be nonzero linear forms on $\pi$ and $\pi^{\vee}$ that are $(H,\chiup)$ and $(H, \chiup^{\vee})$ invariant, respectively, that is,
\begin{align}
l ( \pi (h) v)  &=  \chiup (h) l (v),\quad \text{for all } v \in V;\label{formula::l} \\
l' (\pi^{\vee} (h) \varphi) &= \chiup^{\vee} (h) l'(\varphi),\quad \text{for all }\varphi \in V^{\vee}.\label{formula::l'}
\end{align}
Then the relative character $B = B_{\pi,l,l'}$ satisfies
\begin{align}
L_h B (f)  &=  \chiup^{\vee} (h^{-1}) B (f); \\
R_h B (f)  &=  \chiup (h^{-1}) B (f). 
\end{align}
\end{lem}
\begin{proof}
This follows from a routine calculation, so we omit the details.
\end{proof}

The following proposition, slightly general than Lemma 4 in \cite{Prasad-SelfDual}, holds with exactly the same proof. The author thanks D. Prasad for providing us with the proof.
\begin{prop}\label{prop::Prasad}
Let $H$, $\chiup$ and $\chiup^{\vee}$ be as above. Assume that there is an involution $\iota$ on $G$ such that any distribution $D$ on $G$ satisfying
\begin{align}\label{formula::Distribution--Invariant}
L_{h_1}R_{h_2} D = \chiup^{\vee} (h_1^{-1}) \chiup (h_2^{-1} )  D
\end{align}
is invariant under the anti-involution $g \ra \iota (g^{-1})$, then for every irreducible admissible representation $\pi$ such that $\pi$ is $(H,\chiup)$-distinguished and that $\pi^{\vee}$ is $(H,\chiup^{\vee})$-distinguished, we have
\begin{align}
\pi^{\vee}    \cong  \pi^{\iota}.
\end{align}
\end{prop}
\begin{proof}
For such $\pi$, we can find linear forms $l$ and $l'$ as in \eqref{formula::l} and \eqref{formula::l'}, and form a relative character $B_{\pi,l,l'}$, which is a distribution satisfying \eqref{formula::Distribution--Invariant} by Lemma \ref{lem::2}.

 Note that $l$ and $l'$ can be viewed naturally as linear forms on $\pi^{\iota}$ and $(\pi^{\iota})^{\vee}$. Thus we have a relative character $B_{\pi^{\iota},l,l'}$ of $\pi^{\iota}$ defined by
\begin{align*}
  B_{\pi^{\iota},l,l'} (f)  =  l^{\prime} (\pi^{\iota} (f) l),
\end{align*}
for all $f \in C_c^{\infty} (G)$. One sees easily that
\begin{align}
  B_{\pi^{\iota},l,l'} (f)  =   B_{\pi,l,l'} (f^{\iota}),\quad  f \in C_c^{\infty} (G),
\end{align}
where $f^{\iota}$ is the function defined by $f^{\iota}(g) = f(\iota(g))$.

Reversing the roles of $\pi$ and $\pi^{\vee}$, we have a relative character $B_{\pi^{\vee},l',l}$ of $\pi^{\vee}$ defined by
\begin{align*}
  B_{\pi^{\vee},l',l} (f)  =  l(\pi^{\vee} (f) l^{\prime}),
\end{align*}
for all $f \in C_c^{\infty} (G)$. We claim that
\begin{align}\label{formula:: Comparison-Contra}
  B_{\pi^{\vee},l',l} (f)  =  B_{\pi,l,l'} (f^{\vee}), \quad  f \in C_c^{\infty} (G),
\end{align}
where $f^{\vee}$ is the function defined by $f^{\vee} (g) = f (g^{-1})$. If so, by our assumption for $B_{\pi,l,l^{\prime}}$, we then have $B_{\pi,l,l'} (f^{\iota})  =  B_{\pi,l,l'} (f^{\vee})$. Thus $\pi^{\iota}$ and $\pi^{\vee}$ have a common relative character, therefore they are isomorphic by Lemma \ref{lem::relative charac-}.

It remains to prove the identity \eqref{formula:: Comparison-Contra}, that is,
    \begin{align}\label{formula::goal}
      l(\pi^{\vee} (f) l^{\prime})  =  l^{\prime} ( \pi (f^{\vee})  l).
    \end{align}
 If $l$ and $l^{\prime}$ are smooth linear forms, the identity is clear as it amounts to
\begin{align*}
  \langle v_0, \pi^{\vee}(f) v_0^{\prime} \rangle  =  \langle \pi (f^{\vee}) v_0 , v_0^{\prime}  \rangle,
\end{align*}
where $\langle \cdot, \cdot \rangle$ is the natural bilinear form $\langle \cdot, \cdot \rangle  \colon  \pi \times \pi^{\vee}  \ra \BC$.

In general, choose a compact open subgroup $K \subset G$ such that $e_{K}$, the characteristic function of $K$, has the property that
\begin{align}\label{formula::smooth}
  e_{K} \ast f  =  f  \ast e_K  =f.
\end{align}
Observe that $C_c^{\infty}(G)$ operates both on $\pi$ and the space of linear forms on $\pi$ such that
\begin{align}\label{formula::action}
  ( \pi(\varphi) l ) (v)  =  l ( \pi(\varphi^{\vee})  v  ). 
\end{align}
By \eqref{formula::smooth} and \eqref{formula::action}, and \eqref{formula::goal} for $l$, $l^{\prime}$ replaced by $e_K l$ and $e_K l^{\prime}$, the identity \eqref{formula::goal} is true for general $l$, $l^{\prime}$.

%-------------------  
%	 There seems to be certain subtles when establishing \eqref{formula:: Comparison-Contra}, so we %write down the details. By our assumption, $l \in \pi^{\vee}$ and $l' \in \pi^{\vee \vee} \cong \pi$. %Let $v_{l'} \in \pi$ corresponds to $l'$. Then
%	 \begin{align*}
%	 &B_{\pi^{\vee},l',l} (f) =  l(\pi^{\vee} (f) l') = \pi^{\vee}(f) l' (l) \\
%	 &               = \int_G f(g) l'(\pi^{\vee} (g^{-1})l)  dg  = \int_G f (g) l(g v_{l'}) dg.
%	 \end{align*}
%	 On the other hand,
%	 \begin{align*}
%	 &B_{\pi,l,l'}  (f^{\vee})  =  l'(\pi (f^{\vee}) l)  =  \pi (f^{\vee})l (v_{l'})  \\
%	 &     =   \int_G  f(g^{-1}) l (g^{-1}v_{l'}) dg   = \xi_2^{-1} \int_G f (g) l (g v_{l'}) dg.
%	 \end{align*}
%---------------------
\end{proof}

Now we turn to the specific case of generic representations. Let $G$ be a quasi-split reductive group with an $F$-pinning $\mathcal{P}$. Let $\psi_{\mathcal{P}}$ be the associated nondegenerate character on $U$, and $\iota_{G,\mathcal{P}}$ be the associated duality involution. In view of Remark \ref{rem::Involution=Shalika}, one of the main results in \cite{Shalika-MO} is the following theorem.

\begin{thm}[Shalika]\label{theorem::Shalika}
 Suppose that $F$ is a $p$-adic field. Let $D$ be a distribution on $G$ satisfying $Lu_1 Ru_2^{-1} D  =  \psi_{\mathcal{P}} (u_1u_2) D$, $u_1$, $u_2 \in U$. Set $\theta (g)  =  \iota_{G,\mathcal{P}} (g^{-1})$. Then
\begin{align}
D^{\theta}   =   D.
\end{align}
\end{thm}

\begin{cor}
Let $\pi$ be an irreducible admissible representation of $G$. Assume that $\pi$ is $\psi_{\mathcal{P}}$-generic. Then
\begin{align}
 \pi^{\iota_{G,\mathcal{P}}} \cong \pi^{\vee}.
\end{align}
\end{cor}
\begin{proof}
	This follows directly from Proposition \ref{prop::Prasad} and Theorem \ref{theorem::Shalika}.
	\end{proof}
As a special case of the corollary, we see that \emph{all generic representations of $G_2$, $F_4$ or $E_8$ are self-dual.}

\section*{Acknowledgements}

The author thanks Dipendra Prasad for many useful comments and suggestions for improving the paper. In particular, the last section grows out of his suggestion of using Lemma 4 in \cite{Prasad-SelfDual} to solve his conjecture on the contragredient and the duality involution. The author thanks the anonymous referee for his/her comments on the paper. Part of the work was done when the author attended ``Automorphic Representations and Langlands program'' summer school held in Soochow University. It is a pleasure for him to thank the organizers, Zhifeng Peng and Chung Pang Mok, for their warm hospitality.

	\bibliographystyle{alphanum}
	\bibliography{references}
	
\end{document}